\documentclass[11pt,letterpaper]{article}

\usepackage[margin=1in]{geometry}

\usepackage{amsmath,amsthm,amssymb}
\usepackage{mathtools}
\usepackage{mathrsfs}
\usepackage{graphicx}
\usepackage{float}
\usepackage{bm}

\usepackage{hyperref}

\usepackage{tikz}

\theoremstyle{definition}
\newtheorem{definition}{Definition}[section]

\newtheorem{construction}[definition]{Construction}

\newtheorem{problem}[definition]{Problem}

\theoremstyle{plain}

\newtheorem{theorem}[definition]{Theorem}
\newtheorem{lemma}[definition]{Lemma}
\newtheorem{proposition}[definition]{Proposition}
\newtheorem{claim}[definition]{Claim}

\def \l {\ell}
\def \L {L}

\def \C {\mathcal{C}}
\newcommand{\Cm}[1][\l]{\mathcal{C}^-_{#1}}
\newcommand{\fCm}[1][\le\L]{\bm{\mathcal{C}}^-_{#1}}
\newcommand{\ifCm}{\bm{\mathcal{C}}^-}
\newcommand{\ifB}{\bm{\mathcal{B}}}
\newcommand{\fB}[1][\L]{\bm{\mathcal{B}}_{\le #1}}

\newcommand{\fT}{\bm{T}}

\def \F {\mathcal{F}}
\def \fF {\bm{\mathcal{F}}}
\def \K {\mathcal{K}}
\def \E {\mathcal{E}}

\def \mP {\mathcal{P}}

\def \mH {\mathcal{H}}
\def \mHp {\mH'}
\def \fmH {\bm{\mathcal{H}}}

\def \R {\mathbb{R}}
\def \Z {\mathbb{Z}}

\def \tri{\triangle}
\def \ex {\mathrm{ex}}
\def \sm {\setminus}
\def \ce {\coloneqq}

\renewcommand{\le}{\leqslant}
\renewcommand{\ge}{\geqslant}

\renewcommand{\geq}{\geqslant}

\def \eps {\varepsilon}
\def \es {\varnothing}
\renewcommand \b[2] {\binom{#1}{#2}}

\title{Tur{\'a}n density of long tight cycle \\minus one hyperedge}

\author{
J\'ozsef Balogh\thanks{Department of Mathematics, University of Illinois at Urbana-Champaign, Urbana, Illinois 61801, USA. E-mail: \texttt{\{jobal, haoranl8\}@illinois.edu}.} \thanks{Research is partially supported by NSF grants DMS-1764123, RTG DMS-1937241 and FRG DMS-2152488, the Arnold O. Beckman Research Award (UIUC Campus Research Board RB 22000), and the Langan Scholar Fund (UIUC).}
\and
Haoran Luo\footnotemark[1] \thanks{Research is partially supported by FRG DMS-2152488, the Arnold O. Beckman Research Award (UIUC Campus Research Board RB 22000).}
}
\date{}

\begin{document}
\maketitle

\begin{abstract}
Denote by $\Cm$ the $3$-uniform hypergraph obtained by removing one hyperedge from the tight cycle on $\l$ vertices.
It is conjectured that the Tur\'an density of $\Cm[5]$ is $1/4$.
In this paper, we make progress toward this conjecture by proving that the Tur\'an density of $\Cm$ is $1/4$,
for every sufficiently large $\l$ not divisible by $3$.
One of the main ingredients of our proof is a forbidden-subhypergraph characterization of the hypergraphs, for which there exists a tournament on the same vertex set such that every hyperedge is a cyclic triangle in this tournament.

A byproduct  of our method is a human-checkable proof for the upper bound on the maximum number of almost similar triangles in a planar point set, which was recently proved using the method of flag algebras by Balogh, Clemen, and Lidick\'y.
\end{abstract}

\section{Introduction} \label{sec::Int}
For a collection $\fF$ of $r$-uniform hypergraphs (\emph{$r$-graphs}), we say that an $r$-graph $\mH$ is \emph{$\fF$-free} or \emph{free of $\fF$}, if $\mH$ contains no $\F \in \fF$ as a subhypergraph. The \emph{Tur\'an number} $\ex(n,\fF)$ is defined to be the maximum number of $r$-edges an $n$-vertex $\fF$-free $r$-graph can have. To determine $\ex(n,\fF)$ is a central problem in Extremal Combinatorics, but also notoriously hard when $r \ge 3$, where even the \emph{Tur\'an density} $\pi(\fF) \ce \lim_{n\to \infty} \ex(n,\fF) / \b{n}{r}$ is only known for a few $\fF$'s.\footnote{When $\fF = \{\F\}$, we use $\ex(n,\F)$ and $\pi(\F)$ for $\ex(n,\{\F\})$ and $\pi(\{\F\})$, respectively.} For example, let $\K_4$ be the complete $3$-graph on $4$ vertices and $\K_4^-$ be the $3$-graph obtained by removing one hyperedge from $\K_4$. It has been a long-standing open problem to determine $\pi(\K_4)$ and $\pi(\K_4^-)$.

A related family that has also received extensive attention for Tur\'an density problems is the tight cycles. For every integer $\l \ge 4$,
let $\C_\l$ be the \emph{tight cycle} of size $\l$, i.e., it has vertex set $\{0,1,\ldots, \l-1\}$ and hyperedges $\{ \{i,i+1,i+2 \pmod{\l}\}: 0\le i \le \l-1\}$, and let $\Cm$ be the \emph{tight cycle minus one hyperedge} of size $\l$, i.e., it is obtained from $\C_\l$ by removing the hyperedge $\{\l-1, 0, 1\}$. Note that $\C_4 = \K_4$ and $\Cm[4] = \K_4^-$. When $\l$ is a multiple of $3$, $\C_\l$ is tripartite, so, by a classical result of Erd\H{o}s~\cite{erdos1964extremal}, $\pi(\Cm) = \pi(\C_\l) = 0$.
It is conjectured that $\pi(\C_5) = 2\sqrt{3} -3$ and $\pi(\Cm[5]) = 1/4$, see ~\cite{mubayi2011hypergraph}.
In this paper, we make progress toward the latter conjecture by proving the following theorem.

\begin{theorem} \label{thm::main}
There is a constant $\L$ such that $\pi(\Cm) = 1/4$, for every $\l > L$ not divisible by $3$.
\end{theorem}

Our key method for proving Theorem~\ref{thm::main} is to reduce this hypergraph Tur\'an problem to a counting problem in tournaments, which is in general much easier to deal with than hypergraphs.
We note that a similar framework is used by Kam{\v{c}}ev, Letzter, and Pokrovskiy~\cite{kamvcev2022turan} for the Tur\'an density of $\C_\l$ for sufficiently large $\l$ not divisible by $3$. Some of our ideas and lemmas are partially inspired by them. We also note that Piga, Sales, and Sch{\"u}lke~\cite{piga2022codegree} recently proved that the codegree Tur\'an density of $\Cm$ is $0$, for every $\l \ge 5$.

We now give an outline of the proof for Theorem~\ref{thm::main}.
The lower bound in Theorem~\ref{thm::main} follows from the following construction, which is usually called the iterated blow-up of a hyperedge. It is also conjectured to be the extremal construction for $\Cm[5]$, see Section~2.5 in~\cite{mubayi2011hypergraph}.
\begin{construction} \label{con::iteEdge}
Define $3$-graphs $\E_n$ by induction. The vertex set of $\E_n$ is $\{1,2,\ldots, n\}$, which is partitioned into three parts $V_1,V_2,V_3$ with sizes $\lfloor n/3 \rfloor$, $\lfloor(n+1)/3 \rfloor$, and $\lfloor (n+2)/3 \rfloor$, respectively. $\E_n$ contains no hyperedge when $n = 1$ or $2$. For $n \ge 3$, $\E_n$ contains all the hyperedges with exactly one vertex in each $V_i$, and $V_i$ spans a copy of $\E_{|V_i|}$ for $1\le i \le 3$.
\end{construction}
It is easy to check that $\Cm \not \subseteq \E_n$, when $\l \ge 4$ and $3 \nmid \l$.
A standard induction shows that $\E_n$ has at least $n^3/24 - Cn \log n$ hyperedges for some constant $C> 0$, see also Section 1 in~\cite{barany2018almost}.

For the upper bound in Theorem~\ref{thm::main}, we will first work on the Tur\'an problem of \emph{pseudo-cycles}, which are, roughly speaking, tight cycles with repeated vertices allowed. See Definition~\ref{def::pseCyc} for a rigorous definition.
The key step is to connect this problem with counting the number of cyclic triangles in tournaments. We introduce the following notion.
\begin{definition}
A $3$-graph $\mH$ is \emph{orientable} if there is a tournament $T$ on the same vertex set such that every hyperedge in $\mH$ is a cyclic triangle in $T$.
\end{definition}
For example, it can be checked that $\C_5$ is orientable, but $\K_4^-$ is not.
We remark that the connection between $3$-graphs and tournaments has been noticed decades ago, which can be traced back to the work of Erd\H{o}s and Hajnal~\cite{erdos1972ramsey} in 1972.
For example, orientable $3$-graphs serve as the constructions for the lower bound of the codegree Tur\'an density of $\K_4^-$, see~\cite{falgas2023codegree}, and
the uniform Tur\'an density of $\K_4^-$, see~\cite{glebov2016problem} and~\cite{reiher2018turan}.
See also~\cite{reiher2018generalisation} for a generalization of orientable $3$-graphs to $r$-graphs with $r \ge 4$.

We will prove that a $3$-graph $\mH$ is orientable if and only if it is free of a certain family of hypergraphs, which we call \emph{bottles}, see Proposition~\ref{pro::ori}. Using this characterization of orientable $3$-graphs, we prove that a $3$-graph is orientable if it is free of all the pseudo-cycles minus one hyperedge with length not divisible by $3$, see Lemma~\ref{lem::freifCmOri}.
Then, by analyzing tournaments, we are able to prove a stability result for $\Cm[5]$-free orientable hypergraphs, which says that the vertex set of an almost maximum $\Cm[5]$-free orientable $3$-graph can be partitioned into three parts with almost equal sizes such that there are very few \emph{bad} hyperedges, i.e., the hyperedges with two vertices in a part and one vertex in another part, see Proposition~\ref{pro::sta}.
Building on this structure, a cleaning argument in Section~\ref{sec::Pro} shows that the maximum $3$-graphs free of the pseudo-cycles minus one hyperedge with length not divisible by $3$ and less than a fixed large constant indeed contain no bad hyperedge, from which we can easily prove that such hypergraphs have edge density at most $1/4+ o(1)$.
Finally, a standard technique using blow-ups gives Theorem~\ref{thm::main}.

As a direct application of Theorem~\ref{thm::main}, we give a human-checkable answer to the following question about the maximum number of almost similar triangles in a planar point set. For a triangle $\tri$ with angles $0<a_1 \le a_2 \le a_3<180^\circ$, we say that another triangle $\tri'$ with angles $0<a_1' \le a_2'\le a_3'<180^\circ$ is \emph{$\eps$-similar} to $\tri$ if $|a_i - a_i'| \le \eps$ for $i = 1,2,3$. Inspired by the work of Conway, Croft, Erd\H{o}s, and Guy~\cite{conway1979distribution} about the distribution of angles determined by a planar set, B\'ar\'any and F\"uredi~\cite{barany2018almost} studied $h(n,\tri,\eps)$, the maximum number of triangles that are $\eps$-similar to $\tri$ in a planar set of $n$ points.
They~\cite{barany2018almost} proved that $h(\tri,\eps)\ce \lim_{n\to \infty} h(n,\tri,\eps)/n^3$ exists and it is at least $1/24$ for every triangle $\tri$ and $\eps >0$, see their Figure 1.
They also showed that $h(\tri,\eps) = 1/24$, when $\tri$ is the equilateral triangle and $\eps \le 1^\circ$, and $h(\tri,\eps)$ can be strictly larger for some $\tri$'s, including all the right-angled triangles.
In order to give a general upper bound for $h(n,\tri,\eps)$, they represented the shape of triangles by points in $S_{tri} \ce \{ (a_1,a_2,a_3) \in \mathbb{R}^3: a_1,a_2,a_3 > 0,\, a_1+a_2+a_3 = \pi\}$ and considered the Lebesgue measure on $S_{tri}$. In the same paper, they showed that for almost all triangles $\tri$, there exists $\eps = \eps(\tri)>0$ such that $h(n,\tri, \eps) \le 0.25108\b{n}{3}(1+o(1))$, and with the further aid of the flag algebra method developed by Razborov~\cite{razborov2007flag}, they could improve the constant to $0.25072$. Their main idea is to reduce this problem to a hypergraph Tur\'an problem, by noticing that there exists a family $\fF_{tri}$ of $3$-graphs, whose hyperedges cannot be represented by triangles $\eps$-similar to $\tri$ in any planar set, for almost all triangles $\tri$. See Definition~\ref{def::fFtri} for the rigorous definition of $\fF_{tri}$.
Extending this idea, Balogh, Clemen, and Lidick\'y~\cite{balogh2022maximum} improved this bound to $0.25$, which is best possible, by verifying that more $3$-graphs are members in $\fF_{tri}$ and using flag algebra and the stability method.
\begin{theorem}[Theorem 1.3 in~\cite{balogh2022maximum}] \label{thm::triangle}
For almost all triangles $\tri$, there exists $\eps = \eps(\tri)>0$ such that $h(\tri,\eps) \le 1/4$.
\end{theorem}
We will show in Section~\ref{sec::tri} that Theorem~\ref{thm::main} implies Theorem~\ref{thm::triangle}, by the observation that $\Cm \in \fF_{tri}$ for some large $\l$ not divisible by $3$.

The rest of this paper is organized as follows. In Section~\ref{sec::Pre}, we introduce our notation and lemmas used in our proof. In Section~\ref{sec::Cyc}, we give our forbidden-subhypergraph characterization of the orientable $3$-graphs and prove several other lemmas about tournaments. In Section~\ref{sec::Sta}, we prove our stability result. In Section~\ref{sec::Pro}, we prove Theorem~\ref{thm::main}. In Section~\ref{sec::tri}, we give our new proof for Theorem~\ref{thm::triangle}.

\section{Preliminaries} \label{sec::Pre}
For a positive integer $n$, we write $[n]$ for the set $\{1,2,\ldots, n\}$. For a set $X$ and a positive integer $k$, denote by $\b{X}{k}$ the collection of all subsets of $X$ of size $k$.
For sets $X_1, X_2, \ldots, X_k$, let $[X_1,X_2, \ldots, X_k] \ce \{ \{x_1,x_2,\ldots, x_k\} : x_i \in X_i \textrm{ for $1 \le i \le k$}\}$.

For an $r$-graph $\mH$, we use $V(\mH)$ for its vertex set and use $\mH$ to stand for its $r$-edges. In particular, $|\mH|$ denotes the number of $r$-edges in $\mH$.

Let $\mH$ be a $3$-graph.
For a vertex $v \in V(\mH)$ and two (not necessarily disjoint) sets $S_1, S_2\subseteq V(\mH)$, let $N^\mH_{S_1,S_2}(v) \ce \{ \{x,y\}: \{v,x,y\} \in \mH,\,x \in S_1,\,y \in S_2 \}$ be the \emph{link graph} of $v$ between $S_1$ and $S_2$ and $d^\mH_{S_1,S_2}(v) \ce |N^\mH_{S_1,S_2}(v)|$ be the \emph{degree} of $v$ between $S_1$ and $S_2$. When $S_1 = S_2 =S$, we write $N^\mH_{S}(v)$ for $N^{\mH}_{S_1,S_2}(v)$ and $d^\mH_S(v)$ for $d^\mH_{S_1,S_2}(v)$. Let $N^\mH(v) \ce N^\mH_{V(\mH)}(v)$ and $d^\mH(v) \ce d^\mH_{V(\mH)}(v)$.
For vertices $u, v \in V(\mH)$ and a set $S \subseteq V(\mH)$, let $N^\mH_S(u,v) \ce \{w \in S: \{u,v,w\} \in \mH\}$ be the set of \emph{neighbors} of $u,v$ in $S$ and $d^\mH_S(u,v) \ce |N^\mH_S(u,v)|$ be the \emph{codegree} of $u,v$ in $S$. Let $N^\mH(u,v) \ce N^\mH_{V(\mH)}(u,v)$ and $d^\mH(u,v) \ce d^\mH_{V(\mH)}(u,v)$. We often omit the superscript $\mH$ when it is clear from the context.
For vertex sets $S_1, S_2, S_3 \subseteq V(\mH)$, let  $\mH[S_1,S_2,S_3] \ce \mH \cap [S_1,S_2,S_3]$ be the set of hyperedges between $S_1,S_2,S_3$. Let $\mH[S_1] \ce \mH[S_1, S_1, S_1]$ and $\bar{\mH}[S_1,S_2,S_3] \ce [S_1,S_2,S_3] \sm \mH[S_1, S_2, S_3]$.
For a partition $\pi = (S_1,S_2,S_3)$ of $V(\mH)$, we write $\mH_\pi \ce \mH[S_1,S_2,S_3]$ and $\bar{\mH}_\pi \ce [S_1,S_2,S_3] \sm \mH_\pi$.

For a tournament $T$ and a vertex $v \in V(T)$, let $N^+(v) \ce \{ u \in V(G) : v \to u\}$ and let $N^-(v) \ce \{u \in V(G) : u \to v\}$ be the sets of \emph{out-neighbors} and \emph{in-neighbors} of $v$, respectively. Let $d^+(v) \ce |N^+(v)|$ and $d^-(v) \ce |N^-(v)|$ be the \emph{out-degree} and \emph{in-degree} of $v$, respectively. For vertex sets $V_1, V_2 \subseteq V(T)$, we write $V_1 \to V_2$ if $u \to v$ for every $u \in V_1$ and $v \in V_2$.

We use the following two results about triangle-free graphs.
\begin{theorem}[Mantel~\cite{mantel1907Pro}] \label{thm::Mantel}
Every $n$-vertex triangle-free graph has at most $n^2/4$ edges.
\end{theorem}

\begin{theorem}[Erd\H{o}s, Faudree, Pach, and Spencer, Theorem 1 in~\cite{erdos1988make}] \label{thm::manSta}
Every triangle-free graph $G$ with $n$ vertices and $m$ edges can be made bipartite by removing at most
$$
\min \left\{
\frac{m}{2} - \frac{2m(2m^2 - n^3)}{n^2(n^2-2m)}
,
m - \frac{4m^2}{n^2}
\right\}
$$
edges.
\end{theorem}

For a $3$-graph $\mH$ and a positive integer $t$, the \emph{$t$-blow-up} $\mH[t]$ is the $3$-graph on vertex set $V(\mH) \times [t]$
with hyperedges $\{ \{(v_1,t_1),(v_2,t_2),(v_3,t_3)\}: \{v_1,v_2,v_3\} \in \mH, 1\le t_1,t_2,t_3 \le t  \}$.
For a family of $3$-graphs $\fmH = \{\mH_1,\ldots,\mH_h\}$, let $\fmH[t] \ce \{\mH_1[t],\ldots,\mH_h[t]\}$.

\begin{theorem}[See Section 2 in~\cite{keevash2011hypergraph}] \label{thm::turanBlowup} For every family of $3$-graphs $\fmH$ and positive integer $t$, we have
$\pi(\fmH[t]) = \pi(\fmH)$.
\end{theorem}

We use the following version of the removal lemma for tournaments by Choi, Lidick\'{y}, and Pfender~\cite{choi2020inducibility} (see their Lemma 5; their original theorem is about oriented graphs, and we can, for example, add the graph consisting of two isolated vertices to the forbidden family to get the following version), which follows from a general theorem by Aroskar and Cummings~\cite{aroskar2014limits}.

\begin{theorem} \label{thm::remLem}
Let $\fT$ be a (possibly infinite) set of tournaments. For every $\eps>0$, there exist $n_0$ and $\delta>0$ such that for every tournament $T$ on $n \geq n_0$ vertices, if $T$ contains at most $\delta n^{|V(D)|}$ copies of $D$ for each $D$ in $\fT$, then there exists $T'$ on the same vertex set such that $T'$ is $D$-free for every $D$ in $\fT$ and $T'$ can be obtained from $T$ by reorienting at most $\eps n^2$ edges.
\end{theorem}

The following is a technical lemma that will be used in our proof of the stability result in Section~\ref{sec::Sta}. We omit its standard proof.

\begin{lemma} \label{lem::pro}
For integers $a \ge b > 0$, we have
$$
\max \left(\sum_{i = 1}^\infty x_iy_i \right)
\le
\left \lfloor \frac{a}{b} \right \rfloor \cdot \frac{b^2}{4}
+
\frac{1}{4}\left(a - b \left \lfloor \frac{a}{b} \right \rfloor \right)^2
,
$$
where the maximum is over non-negative integer sequences $(x_1,x_2,\ldots)$ and $(y_1,y_2,\ldots)$ such that $\sum_{i=1}^\infty (x_i+y_i) = a$ and $x_j + y_j \le b$ for every $j \ge 1$.
\end{lemma}

\section{Cyclic triangles in tournaments} \label{sec::Cyc}
In this section, we provide the first step toward proving Theorem~\ref{thm::main}, by relating it to the number of cyclic triangles in tournaments.

For $3$-graphs $\mH$ and $\F$, a \emph{surjective homomorphism} from $\mH$ to $\F$ is a surjective map from $V(\mH)$ to $V(\F)$ such that $\{f(v_1),f(v_2),f(v_3)\} \in \F$ if $\{v_1,v_2,v_3\} \in \mH$ and for every $\{u_1, u_2, u_3\} \in \F$, there is $\{v_1,v_2,v_3\} \in \mH$ with $\{f(v_1),f(v_2),f(v_3)\} = \{u_1, u_2, u_3\}$.
Following~\cite{kamvcev2022turan}, we have the following definitions.
\begin{definition} \label{def::psePath}
For every integer $\l \ge 3$, let $\mP_\l$ be the path of size $\l$, i.e., it is the $3$-graph on vertex set $\{1,2,\ldots, \l\}$ with hyperedges $\{\{i,i+1,i+2\}: 1 \le i \le \l -2\}$.
We call a $3$-graph $\mH$ a \emph{pseudo-path} of size $\l$ if there exists a surjective homomorphism from $\mP_\l$ to $\mH$.
\end{definition}
Hence, pseudo-paths are a generalization of paths, where repeated vertices are allowed. For tight cycles, we have the following similar notion.

\begin{definition} \label{def::pseCyc}
For every integer $\l \ge 4$, we call a $3$-graph $\mH$ a \emph{pseudo-cycle} of size $\l$ if there exists a surjective homomorphism from $\C_\l$ to $\mH$, and
we call a $3$-graph $\mH$ a \emph{pseudo-cycle minus one hyperedge} of size $\l$ if there exists a surjective homomorphism from $\Cm$ to $\mH$.
\end{definition}
For every integer $\L \ge 4$, let $\fCm$ be the set of all the pseudo-cycles minus one hyperedge of size $\l$, where $4 \le \l \le \L$ and $3 \nmid \l$. Let $\ifCm \ce \bigcup_{\L \ge 4} \fCm$. It can be easily checked that $\E_n$, the iterated blow-up of a hyperedge defined in Construction~\ref{con::iteEdge}, is $\ifCm$-free.

For a copy $\mH$ of pseudo-path, pseudo-cycle minus one hyperedge, or a pseudo-cycle of size $\l$, we often use $v_0v_2\ldots v_{\l-1}$, a sequence of its vertices (with repetition allowed) to stand for it; this means that $\mH$ consists of hyperedges $\{v_{i},v_{i+1 \pmod{\l}},v_{i+2 \pmod{\l}} \}$ for $0 \le i \le \l-3$, when $\mH$ is a pseudo-path, for $0 \le i \le \l-2$, when $\mH$ is a pseudo-cycle minus one hyperedge, and for $0\le i \le \l-1$, when $\mH$ is a pseudo-cycle.

For $k \ge 4$, we call a pseudo-path $v_1v_2\ldots v_k v_2v_1$ a \emph{bottle} of size $k+2$, see Figure~\ref{fig::bottle}.
For $\L \ge 6$,
let $\fB$ be the set of the bottles of size $\l$ where $6 \le \l \le \L$ and $\ifB \ce \bigcup_{\L \ge 6} \fB$.

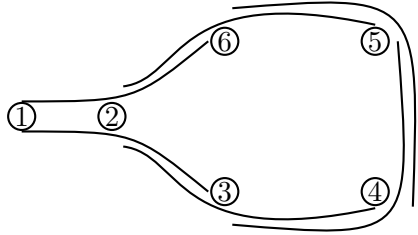
\begin{figure}[ht]
    \centering
    \begin{tikzpicture}
    \draw [thick] (-2.7,0) circle [radius=0.18] node {$1$};
    \draw [thick](-1.5,0) circle [radius=0.18] node {$2$};
    \draw [thick](0,-1) circle [radius=0.18] node {$3$};
    \draw [thick](2,-1) circle [radius=0.18] node {$4$};
    \draw [thick](2,1) circle [radius=0.18] node {$5$};
    \draw [thick](0,1) circle [radius=0.18] node {$6$};
    \draw [thick] (-2.7,-0.2) .. controls (-1.2,-0.2) .. (-0.22,-1.0);
    \draw [thick](-2.7, 0.2) .. controls (-1.2, 0.2) .. (-0.22, 1.0);
    \draw [thick](-1.35, -0.4) .. controls (-0.55,-0.5) and (-0.7, -1.8) .. (2, -1.22);
    \draw [thick](-1.35,  0.4) .. controls (-0.55, 0.5) and (-0.7,  1.8) .. (2,  1.22);
    \draw [thick](0.1,-1.44) .. controls (2.5,-1.64) .. (2.3, 1);
    \draw [thick] (0.1, 1.44) .. controls (2.6, 1.64) .. (2.5, -1.2);
    \end{tikzpicture}
    \caption{An example of a bottle of size $8$. It has vertex set $\{1,2,\ldots, 6\}$ and hyperedges $\{1,2,3\},\{2,3,4\},\{3,4,5\},\{4,5,6\},\{5,6,2\},\{6,2,1\}$.
    It can be represented as $12345621$.
    }
    \label{fig::bottle}
\end{figure}

\begin{proposition} \label{pro::ori}
A $3$-graph $\mH$ is orientable if and only if it is free of $\ifB$.
\end{proposition}
\begin{proof}
Assume that there is a bottle $v_1v_2\ldots v_k v_2v_1$ in $\mH$, where $k \ge 4$. If $\mH$ is orientable, then without loss of generality, assume that in the corresponding tournament $T$ we have $v_1 \to v_2$. Then, by the definition of the bottle, we have $v_{i-1} \to v_i$ for $1 < i \le k$ and then $v_k \to v_2$ and $v_2 \to v_1$, a contradiction.

Assume that $\mH$ is free of $\ifB$. We say that two pairs of vertices $\{a,b\}$ and $\{c,d\}$ are \emph{tightly connected} if there is a pseudo-path $u_1u_2\ldots u_{\l -1}u_{\l}$ such that $\{u_1,u_2\} = \{a,b\}$ and $\{u_{\l -1}, u_\l\} = \{c,d\}$. Note that the three pairs of vertices in every hyperedge are always tightly connected to each other. Hence, we can partition the hyperedges of $\mH$ into equivalence classes $\mH_1, \mH_2, \ldots, \mH_p$, where hyperedges $\{a,b,c\}$ and $\{x,y,z\}$ are in the same class if $\{a,b\}$ and $\{x,y\}$ are tightly connected.
Note that for every pair of vertices $\{x,y\}$, there can be at most one $i$ such that $\{x,y\}$ is contained in some hyperedges of $\mH_i$, so when trying to orient $\{x,y\}$, we only need to consider hyperedges of $\mH_i$ and omit all others. Let $P_i$ be the set of pairs of vertices contained in some hyperedges of $\mH_i$.

Construct a tournament $T$ on $V(\mH)$ as follows.
For  every $i$, where $1 \le i \le p$, do the following algorithm.
Choose an arbitrary pair $\{a,b\}$ in $P_i$ and orient $a \to b$. For every other pair $\{c,d\}$ in $P_i$, there are two cases.
\begin{enumerate}
    \item There is a pseudo-path $\F_{1} = abw_3 \ldots w_{k_1} cd$ or a pseudo-path $\F_{2} = ba x_3 \ldots x_{k_2} dc$.
    \item There is a pseudo-path $\F_{3} = aby_3 \ldots y_{k_3}dc$ or a pseudo-path $\F_{4} = baz_3 \ldots z_{k_4}cd$.
\end{enumerate}
We claim that exactly one case happens. It is clear that at least one of the cases happens, since $\{a,b\}$ and $\{c,d\}$ are tightly connected. If both cases happen, then we have pseudo-paths $\F_{s_1}$ and $\F_{s_2}$, where $1 \le s_1\le 2$ and $3 \le s_2 \le 4$.

\begin{itemize}
    \item If $s_1 = 1, s_2 = 3$, then we have a bottle $abw_3\ldots w_{k_1}cdy_{k_3} \ldots y_3ba$.
    \item If $s_1 = 1, s_2 = 4$, then we have a bottle $dcw_{k_1}\ldots w_3 ba z_3\ldots z_{k_4} cd$.
    \item If $s_1 = 2, s_2 = 3$, then we have a bottle $cdx_{k_2}\ldots x_3 aby_3\ldots y_{k_3} dc$.
    \item If $s_1 = 2, s_2 = 4$, then we have a bottle $bax_3\ldots x_{k_2}dc z_{k_4}\ldots z_3 ab$.
\end{itemize}
Therefore, exactly one case happens, as claimed.
Now, if the first case happens, we orient $c \to d$; otherwise we orient $d \to c$.
For a hyperedge $\{x,y,z\} \in \mH_i$, if we orient $x \to y$, which means that there is a pseudo-path $ab\ldots xy$ or a pseudo-path $ba \ldots yx$, then we also have $y \to z$ and $z \to x$, since there is a pseudo-path $ab \ldots xyzx$ or a pseudo-path $ba \ldots yxzy$. Therefore, every hyperedge in $\mH_i$ is a cyclic triangle in this orientation.

Finally, for pairs of vertices not in any $P_i$, orient them in an arbitrary way. Every hyperedge in $\mH$ is a cyclic triangle in $T$, so $\mH$ is orientable.
\end{proof}

For an orientable $3$-graph $\mH$, let $T(\mH)$ be a tournament on the same vertex set
such that every hyperedge in $\mH$ is a cyclic triangle in $T(\mH)$.
For a tournament $T$, let $\mH(T)$ be the $3$-graph on the same vertex set whose hyperedges are exactly the cyclic triangles in $T$. Note that $\mH \subseteq \mH(T(\mH))$ by definition, and strict containment can happen.

For a tournament $T$, let $t(T)$ be the number of cyclic triangles in $T$. The following lemma is a well-known upper bound for $t(T)$ by Kendall and Smith~\cite{kendall1940method}. We include its proof.

\begin{lemma} \label{lem::numCycTri}
For every tournament $T$ on $n$ vertices, we have
\begin{align*}
    t(T) \le
\left\{
\begin{array}{ll}
\frac{1}{24}(n^3 - n) & \textrm{if $n$ is odd,} \\
\frac{1}{24}(n^3 - 4n) & \textrm{if $n$ is even.}
\end{array}
\right.
\end{align*}
\end{lemma}
\begin{proof}
For every non-cyclic triangle in $T$, it has exactly one vertex with two out-edges and exactly one vertex with two in-edges. Hence, we have
\begin{align}
t(T)
& = \b{n}{3} - \frac{1}{2}\sum_{v \in V(T)}\left( \b{d^+(v)}{2} + \b{d^-(v)}{2} \right) \notag \\
& = \b{n}{3} - \frac{1}{4} \sum_{v \in V(T)} \left( (d^+(v))^2 + (d^-(v))^2 - (n-1) \right) \label{equ::numCycTri}\\
& \le \b{n}{3} - \frac{1}{4} \sum_{v \in V(T)} \left( \left(\left\lceil \frac{n-1}{2} \right\rceil\right)^2 + \left(\left\lfloor \frac{n-1}{2} \right\rfloor\right)^2 - (n-1)
\right). \notag
\end{align}
The claim then follows via an easy calculation.
\end{proof}
Proposition~\ref{pro::ori} and Lemma~\ref{lem::numCycTri} show that $\pi(\ifB) \le 1/4$. Now we convert this bound to $\ifCm$.

\begin{lemma} \label{lem::freifCmOri}
If a $3$-graph $\mH$ is free of $\fCm$, then $\mH$ is free of $\fB[(\L + 2)]$. In particular, if $\mH$ is free of $\ifCm$, then $\mH$ is orientable.
\end{lemma}
\begin{proof}
Assume for contradiction that $\mH$ contains a bottle $v_1v_2v_3\ldots v_k v_2v_1$, where $k \le \L$.
We have $k \neq 4$, since otherwise $\mH$ contains $\K_4^- = \Cm[4] \in \fCm$.
For every $k > 4$, $\mH$ contains two pseudo-cycles minus one hyperedge whose sizes differ by one: $v_kv_1v_2v_3\ldots v_{k-2}v_{k-1}$ and $v_2v_3\ldots v_{k-1}v_k$. The sizes of these two cycles are both at most $k \le \L$, and at least one of them is not divisible by $3$.
Thus, $\mH$ cannot be free of $\fCm$, a contradiction.
The second claim then follows from Proposition~\ref{pro::ori}.
\end{proof}

By Lemmas~\ref{lem::numCycTri} and~\ref{lem::freifCmOri}, we have $\pi(\ifCm) \le 1/4$.
In order to improve this to get Theorem~\ref{thm::main}, we need to study tournaments more carefully. The following lemmas about tournaments will be used in Section~\ref{sec::Sta} to prove our stability result, Proposition~\ref{pro::sta}.

\begin{lemma} \label{lem::badDegVer}
For every $\eps_1, \eps_2 > 0$, there exists $\delta \ge \eps_1\eps_2^2/2$ such that for every tournament $T$ on $n$ vertices with $t(T) > (1/24 - \delta)n^3$,
for
$$
V'(T) \ce \left\{
v\in V(T):
\frac{n-1}{2} - \eps_2 n
<
d^+(v), d^-(v)
<
\frac{n-1}{2} + \eps_2n
\right\},
$$
we have
$
\left| V'(T)  \right| > (1-\eps_1) n.
$
\end{lemma}
\begin{proof}
Let $\delta = \eps_1\eps_2^2 / 2$.
If the claim is not true, then by \eqref{equ::numCycTri}, we have
\begin{align*}
t(T) < \b{n}{3}
&- \frac{1}{4} (\eps_1 n)
\left(
\left(\frac{n-1}{2} + \eps_2 n\right)^2 + \left(\frac{n-1}{2} - \eps_2 n\right)^2 - (n-1)
\right) \\
&- \frac{1}{4} (1-\eps_1)n
\left( 2 \cdot
\left(\frac{n-1}{2}\right)^2 - (n-1)
\right)
=\left(\frac{1}{24} - \frac{1}{2}\eps_1\eps_2^2\right)n^3 - \frac{1}{24}n
< t(T)
,
\end{align*}
a contradiction.
\end{proof}

\begin{lemma} \label{lem::degWithin}
For every $n$-vertex tournament $T$, where $n > 8$, we have
$$
\left| \left\{u \in V(T): d^+(u) \ge \frac{n}{4} \right\}\right| \ge \frac{n}{4} \;\; \textrm{and} \;\;
\left| \left\{u \in V(T): d^-(u) \ge \frac{n}{4} \right\}\right| \ge \frac{n}{4}.
$$
\end{lemma}
\begin{proof}
If any of these two claims is false, then the number of directed edges in $T$ is at most
$$
\frac{n}{4} \cdot n + \frac{3n}{4} \cdot \frac{n}{4} = \frac{7n^2}{16} < \b{n}{2},
$$
a contradiction.
\end{proof}

\begin{lemma} \label{lem::badEdges}
For every $\eps_1,\eps_2$ such that $\eps_1 /24 > \eps_2 > 0$, there exist $\delta > \eps_1^2 ( \eps_1 - 24\eps_2) /24^3$ and $N$ such that for every tournament $T$ on $n>N$ vertices with $t(T) \ge (1/24 - \delta)n^3$,
for
$$
B_T(\eps_2) \ce \left\{\{u,v\}\in \b{V(T)}{2} : \{u,v\} \textrm{ is in at most $\eps_2 n$ cyclic triangles in $T$}\right\},
$$
we have
$
\left| B_T(\eps_2) \right| < \eps_1 n^2.
$
\end{lemma}
\begin{proof}
Let $\delta > \eps_1^2 ( \eps_1 - 24\eps_2) /24^3$ be the one obtained from Lemma~\ref{lem::badDegVer} when applying it with $\eps_1' = \eps_1 /12$ and $\eps_2' = \eps_1/24 - \eps_2/2$. Let $N$ be sufficiently large such that $\eps_1 N > 24$. Let $T$ be a tournament on $n>N$ vertices with $t(T) \ge (1/24 - \delta)n^3$.

Assume for contradiction that $|B_T(\eps_2)| \ge \eps_1 n^2$.
For every vertex $v \in V(T)$, let $N'(v) \ce \{u \in N^-(v): \{u,v\} \in B_T(\eps_2) \}$
and $S \ce \{v \in V(T) : |N'(v)| \ge \eps_1n / 3\}$. If $|S| \le \eps_1 n /3$, then we have
$$
|B_T(\eps_2)| \le \frac{\eps_1n}{3} \cdot n + \left( 1 - \frac{\eps_1}{3}\right)n \cdot \frac{\eps_1n}{3} < \frac{2}{3} \eps_1 n^2 < \eps_1n^2,
$$
a contradiction. Therefore, $|S| > \eps_1 n/3$ and then by Lemma~\ref{lem::badDegVer}, there exists $v_0 \in S \cap V'(T)$.

Note that for every vertex $u \in N^-(v_0)$ and vertex $w \in N^+(v_0)$, if $w \to u$, then $\{u,v_0,w\}$ forms a cyclic triangle. Hence, for every $u \in N'(v_0)$, we have
$
|N^-(u) \cap N^+(v_0)| \le \eps_2 n,
$
and then,
$$
|N^+(u) \cap N^+(v_0)| \ge d^+(v_0) - \eps_2n.
$$
Since $|N'(v_0)| \ge \eps_1n/3$, we have, by Lemma~\ref{lem::degWithin}, that there are at least $\eps_1n/12$ vertices $u \in N'(v_0)$ such that $|N^+(u) \cap N'(v_0)| \ge \eps_1n /12$. However, for every such vertex $u$, we have
$$
d^+(u) \ge |N^+(u) \cap N'(v_0)| + 1 + |N^+(u) \cap N^+(v_0)| >
\frac{\eps_1}{12} n + d^+(v_0)- \eps_2n = d^+(v_0) + 2\eps'_2n,
$$
so $d^+(u) > (n-1)/2 + \eps'_2 n$ and hence $u \notin V'(T)$.
Thus, we have $|V'|\le (1-\eps_1/12)n = (1-\eps'_1) n$, a contradiction to Lemma~\ref{lem::badDegVer}.
\end{proof}

\section{Stability result} \label{sec::Sta}
In this section, we prove our stability result.
Let $D_5$ be the tournament with vertex set $\{1,2,3,4,5\}$ and directed edges $1 \rightarrow 2$, $1 \rightarrow 3$, $1 \leftarrow 4$, $1 \leftarrow 5$, $2\rightarrow 3$, $2\rightarrow 4$, $2 \rightarrow 5$, $3\rightarrow 4$, $3 \leftarrow 5$, $4\leftarrow 5$, see Figure~\ref{fig::D5}. Let $\fT_5$ be the set of the tournaments $D$ on $5$ vertices such that $\mH(D)$ contains a copy of $\Cm[5]$ as a subhypergraph.
We will apply Theorem~\ref{thm::remLem} with these tournaments as $\fT$, the tournaments to remove, for the proof of our stability result.
For a $3$-graph $\mH$ and a $3$-partition $\pi = (V_1,V_2,V_3)$ of $V(H)$,
we write $\mH^\pi_{bad}$ for $\bigcup_{1\le i \neq j \le 3} \mH[V_i, V_i, V_j]$.

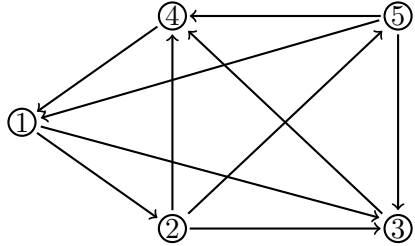
\begin{figure}[ht]
    \centering
    \begin{tikzpicture}
    \draw [thick](-2,0) circle [radius=0.18] node {$1$};
    \draw [thick](0,-1.414) circle [radius=0.18] node {$2$};
    \draw [thick](3,-1.414) circle [radius=0.18] node {$3$};
    \draw [thick](0,1.414) circle [radius=0.18] node {$4$};
    \draw [thick](3,1.414) circle [radius=0.18] node {$5$};
    \draw [thick, ->] (-1.8,-0.1414) -- (-0.2, -1.2726);
    \draw [thick, ->] (-1.75,-0.05656) -- (2.75, -1.27);
    \draw [thick, <-] (-1.8,0.1414) -- (-0.2, 1.2726);
    \draw [thick, <-] (-1.75, 0.05656) -- (2.75, 1.35744);
    \draw [thick, ->] (0, -1.17) -- (0, 1.17);
    \draw [thick, <-] (3, -1.17) -- (3, 1.17);
    \draw [thick, ->] (0.23, -1.414) -- (2.77, -1.414);
    \draw [thick, <-] (0.23,  1.414) -- (2.77,  1.414);
    \draw [thick, ->] (0.21, -1.21604) -- (2.79, 1.21604);
    \draw [thick, <-] (0.21, 1.21604) -- (2.79, -1.21604);
    \end{tikzpicture}
    \caption{The tournament $D_5$.}
    \label{fig::D5}
\end{figure}

\begin{proposition} [Stability result]
\label{pro::sta}
For every $\eps_1, \eps_2 > 0$, there exist $\delta > 0$ and $N$ such that for every $n > N$, the following is true.
For every $n$-vertex $\Cm[5]$-free orientable $3$-graph $\mH$, if $|\mH| > (1/24 - \delta) n^3$, then there exists a $3$-partition $\pi = (V_1,V_2,V_3)$ of $V(\mH)$ such that
$
(1/3 - \eps_2) n < |V_1|,|V_2|,|V_3| < (1/3 + \eps_2) n
$
and
$
|\mH^\pi_{bad}| < \eps_1 n^3.
$
\end{proposition}

\begin{proof}
Given $\eps_1,\eps_2>0$, let $\beta \gg \gamma \gg \delta>0$ be sufficiently small, and let $N$ be sufficiently large. Let $T = T(\mH)$. By assumption, $t(T) \ge |\mH| > (1/24 - \delta) n^3$.

\begin{claim} \label{cla::numT5}
For every $D \in \fT_5$, the number of induced copies of $D$ in $T$ is at most $\delta n^5$.
\end{claim}
\begin{proof}
Since $\mH$ is free of $\Cm[5]$, whenever $T$ contains an induced copy of $D$, this copy must contain a cyclic triangle that is not a hyperedge in $\mH$. By Lemma~\ref{lem::numCycTri}, we have $t(T) < n^3/24$, so the number of such cyclic triangles is at most $\delta n^3$ and hence the number of induced copies of $D$ in $T$ is at most $\delta n^5$.
\end{proof}

\begin{claim} \label{cla::numD5}
The number of induced copies of $D_5$ in $T$ is at most
$\delta^{1/4} n^5$.
\end{claim}
\begin{proof}
Assume that the number of induced copies of $D_5$ is greater than $\delta^{1/4}n^5$. Define
\begin{multline*}
S \ce \{\{v_1,v_2,v_3,v_4,v_5, u\} \in \b{V(\mH)}{6} : v_4 \to u, u \to v_5, \textrm{ and }\\
v_i \to v_j \textrm{ in $T$ iff } i \to j \textrm{ in $D_5$, for } 1\le i < j\le 5\},
\end{multline*}
and define $A_T(\delta^{1/4})$ to be
$$
\left\{ \{v_4,v_5\}:
\left|\left\{\{v_1,v_2,v_3\}: v_i \to v_j \textrm{ in $T$ iff } i \to j \textrm{ in $D_5$, for } 1\le i < j\le 5 \right\}\right| \ge \delta^{1/4} n^3
\right\}.
$$
If $|A_T(\delta^{1/4})| < \delta^{1/4}n^2/4$, then the number of induced copies of $D_5$ in $T$ is at most
$$
\frac{1}{4}\delta^{\frac{1}{4}}n^2 \cdot \b{n}{3}
+ \b{n}{2} \cdot \frac{1}{4}\delta^{\frac{1}{4}}n^3
< \delta^{\frac{1}{4}} n^5,
$$
a contradiction to our assumption, so $|A_T(\delta^{1/4})| \ge \delta^{1/4}n^2/4$.
Since $t(T) \ge |\mH| > (1/24 -\delta)n^3$, by Lemma~\ref{lem::badEdges}, we have $|B_T(\delta^{1/3})| < 100\delta^{1/3}n^2$, so $|A_T(\delta^{1/4}) \sm B_T(\delta^{1/3})| \ge \delta^{1/4}n^2 /5$. Note that $\{4,5\}$ is not in any cyclic triangle in $D_5$, and $\{v_4,v_5,u\}$ forms a cyclic triangle, for every $\{v_1,v_2,v_3,v_4,v_5, u\} \in S$. Hence, every pair in $A_T(\delta^{1/4}) \sm B_T(\delta^{1/3})$ is in at least $\delta^{1/4}n^3 \cdot \delta^{1/3}n$ sets in $S$. Trivially, every set in $S$ contains at most $\b{6}{2} = 15$ pairs in $A_T(\delta^{1/4}) \sm B_T(\delta^{1/3})$.
Therefore,
we have
\begin{equation} \label{equ::lowerBoundS}
|S|
\ge \frac{1}{15} \cdot \delta^{\frac{1}{4}} n^3 \cdot \delta^{\frac{1}{3}}n \cdot |A_T(\delta^{1/4}) \sm B_T(\delta^{1/3})|
\ge \frac{1}{15} \cdot \delta^{\frac{1}{4}} n^3 \cdot \delta^{\frac{1}{3}}n \cdot \frac{1}{5} \delta^{\frac{1}{4}}n^2 = \frac{1}{75}\delta^{\frac{5}{6}}n^6.
\end{equation}

For every $F = \{v_1,v_2,v_3,v_4,v_5,u\} \in S$, consider the orientation of edges between $v_i$ and $u$, for $1\le i \le 3$. There are eight possibilities.
We claim that $F$ contains a copy of some $D \in \fT_5$ in every case.
\begin{itemize}
    \item If $u \gets v_1$, then $v_4uv_5v_1v_2$ forms a copy of $\Cm[5]$ in $\mH(T)$.
    \item If $u \to v_1$ and $u \gets v_3$, then $v_4v_5uv_3v_1$ forms a copy of $\Cm[5]$ in $\mH(T)$.
    \item If $u \to v_1$, $u \to v_2$, and $u \to v_3$, then $v_5v_1v_2v_4u$ forms a copy of $\Cm[5]$ in $\mH(T)$.
    \item If $u \to v_1$, $u \gets v_2$, and $u \to v_3$, then $uv_2v_1v_4v_3$ forms a copy of $\Cm[5]$ in $\mH(T)$.
\end{itemize}
For a fixed $D \in \fT_5$, by Claim~\ref{cla::numT5}, the number of copies of $D$ in $T$ is at most $\delta n^5$, and for every copy of $D$, it can be in at most $n$ sets in $S$. Therefore, we have
\begin{equation} \label{equ::upperBoundS}
|S| \le |\fT_5| \cdot \delta n^5 \cdot n.
\end{equation}
By~\eqref{equ::lowerBoundS} and~\eqref{equ::upperBoundS}, we have $|\fT_5| \ge \frac{1}{75}\delta^{-1/6} > 2^{\b{5}{2}}\ge |\fT_5|$, a contradiction.
\end{proof}
Applying Theorem~\ref{thm::remLem} to $T$ with $\fT = \{D_5\} \cup \fT_5$, we get a tournament $T'$ free of $D_5$, where $\mH(T')$ is free of $\Cm[5]$ and $T'$ and $T$ differ by at most $\gamma n^2/2$ edges. Since changing the orientation of one edge can remove at most $n$ cyclic triangles, we have $t(T') \ge t(T) - \gamma n^3/2 > (1/24 - \gamma) n^3$. Let $\mH' \ce \mH(T')$.

By Lemma~\ref{lem::badDegVer}, there exists $V'\subseteq V(T')= V(\mH)$ with $|V'| = m \ge (1- \beta/2)n$ such that $(1/2-\beta/5)n < d^+(v), d^-(v) < (1/2+\beta/5)n$ in $T'$ for every vertex $v \in V'$. Let $T''$ be the subtournament of $T'$ induced by $V'$. Let $\mH'' \ce \mH(T'')$. We have
\begin{equation} \label{equ::tTpp}
t(T'')
\ge t(T') - \frac{\beta}{2} n \cdot  n^2
> \left(\frac{1}{24} - \gamma - \frac{\beta}{2} \right) n^3
>\left(\frac{1}{24} - \beta \right) n^3
\ge \left(\frac{1}{24} - \beta \right) m^3
,
\end{equation}
and
\begin{equation} \label{equ::Tppdl}
d^+(v), d^-(v)
> \left(\frac{1}{2} - \frac{\beta}{5} \right) n - \frac{\beta}{2} n
> \left(\frac{1}{2} - \beta \right) n
\ge \left(\frac{1}{2} - \beta \right) m,
\end{equation}
\begin{equation} \label{equ::Tppdu}
d^+(v), d^-(v)
< \left( \frac{1}{2} + \frac{\beta}{5} \right) n
\le \left( \frac{\frac{1}{2} + \frac{\beta}{5}}{1- \frac{\beta}{2}} \right) m
< \left(\frac{1}{2} + \beta \right) m
\end{equation}
in $T''$ for every $v \in V'$.

\begin{claim} \label{cla::comBip}
For every vertex $v \in V'$, every component of $N^{\mH''}_{V'}(v)$ is a complete bipartite graph.
\end{claim}
\begin{proof}
For every vertex $v \in V'$, its link graph $N^{\mH''}_{V'}(v)$ is a bipartite graph with the two parts $N^+(v)$ and $N^-(v)$ in $T''$.
We first claim that if there are vertices $w,x,y,z$ such that $w,y \in N^+(v)$, $x,z \in N^-(v)$, and $w \to x$, $x \gets y$, $y \to z$ in $T''$, then we also have $w \to z$ in $T''$. Equivalently, this is to say that whenever $\{w,x\}, \{x,y\}, \{y,z\} \in N^{\mH''}_{V'}(v)$, we also have $\{w,z\} \in N^{\mH''}_{V'}(v)$. Assume for contradiction that $w \gets z$. If $w \to y$, then $wxvyz$ forms a copy of $\Cm[5]$ in $\mH''$. If $x \to z$, then $zyvxw$ forms a copy of $\Cm[5]$ in $\mH''$. Hence, the only possibility is $w \gets y$ and $x \gets z$, but then $\{v, y, w, x,z\}$ forms a copy of $D_5$ in $T''$, still a contradiction.

Now assume for contradiction that there exist vertices $a \in N^+(v)$ and $b \in N^-(v)$ in the same component of $N^{\mH''}_{V'}(v)$ but $\{a,b\} \notin N^{\mH''}_{V'}(v)$. Let $P$ be the shortest path between $a$ and $b$ in $N^{\mH''}_{V'}(v)$. Since $N^{\mH''}_{V'}(v)$ is bipartite and $\{a,b\} \notin N^{\mH''}_{V'}(v)$, we have that $P$ contains at least $4$ vertices, so we can assume that $P$ starts with $a, v_1, u_2,v_3$.
However, by the claim in the last paragraph, we have $\{a,v_3\} \in N^{\mH''}_{V'}(v)$, a contradiction to that $P$ is shortest.
\end{proof}

Let $v_0 \in V'$ be a vertex in maximum number of cyclic triangles in $T''$. Since $t(T'') \ge (1/24 - \beta )m^3$ by~\eqref{equ::tTpp}, we have that $v_0$ is in at least $(1/8 - 3\beta)m^2$ cyclic triangles and hence
\begin{equation} \label{equ::dmHppVpv0}
d^{\mH''}_{V'}(v_0) \ge (1/8 - 3\beta)m^2.
\end{equation}
Assume that $V'_2 \subseteq N^+(v_0), V'_3 \subseteq N^-(v_0)$ form the largest component in $N^{\mH''}_{V'}(v_0)$. Let $V'_{1a} \ce N^+(v_0) \sm V'_2$, $V'_{1b} \ce N^-(v_0) \sm V'_3$, and let $V'_1 \ce \{v_0\} \cup V'_{1a} \cup V'_{1b}$.
Consider the partition $\pi' \ce (V'_1,V'_2,V'_3)$ of $V'$.
By Claim~\ref{cla::comBip}, we have $V'_2 \to V'_3$, and by the definition of $V_2',V_3'$, we have $V'_2 \gets V'_{1b}$ and $V'_{1a} \gets V'_3$,
see Figure~\ref{fig::V'1V'2V'3}.

\begin{figure}[ht]
    \centering
    \begin{tikzpicture}
    \draw [thick](-1.7,0) circle [radius=0.23] node {$v_0$};
    \draw[very thick] (0,0.9) rectangle (3,2.3);
    \draw[very thick] (0,-2.3) rectangle (3,-0.9);
    \draw[very thick] (-2.2,-2.6) rectangle (1,2.6);

    \draw[thick, <-] (-1.5, 0.25 ) -- (-0.15, 1.6);
    \draw[thick, ->] (-1.5, -0.25 ) -- (-0.15, -1.6);
    \draw[thick, ->] (2, -0.8 ) -- (2, 0.8);
    \draw[thick, ->] (1.9, 0.8 ) -- (0.5, -0.8);
    \draw[thick, <-] (1.9, -0.8 ) -- (0.5, 0.8);

    \node at (1.9, 1.6) [fill=black!0,draw=black!0] {$V'_3$};
    \node at (1.9,-1.6) [fill=black!0,draw=black!0] {$V'_2$};
    \node at (0.45, 1.6) [fill=black!0,draw=black!0] {$V'_{1b}$};
    \node at (0.45,-1.6) [fill=black!0,draw=black!0] {$V'_{1a}$};
    \node at (-0.8,2) [fill=black!0,draw=black!0] {$V'_1$};
    \node at (4,  1.6) [fill=black!0,draw=black!0] {$N^-(v_0)$};
    \node at (4, -1.6) [fill=black!0,draw=black!0] {$N^+(v_0)$};
    \end{tikzpicture}
    \caption{The subsets $V'_1,V'_2,V'_3$ of $V'$.}
    \label{fig::V'1V'2V'3}
\end{figure}
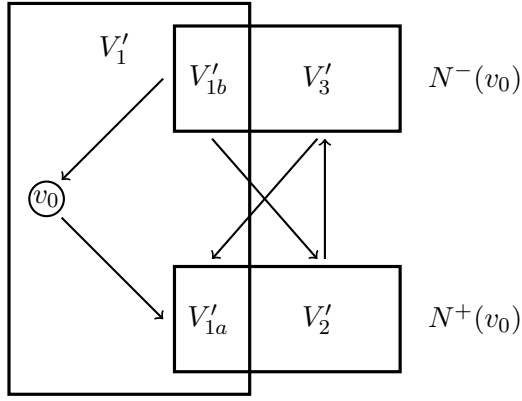

\begin{claim} \label{cla::sizeV'2V'3}
We have $|V'_2| + |V'_3| \ge (1/2 - 8\beta) m$.
\end{claim}

\begin{proof}
By Claim~\ref{cla::comBip}, we can assume that $N^{\mH''}_{V'}(v_0)$ is partitioned into complete bipartite graphs $G_1, \ldots, G_p$ and assume that $G_i$ has two parts with sizes $x_i$ and $y_i$, for $1 \le i \le k$.
Let $a = m-1$ and $b = |V'_2|+|V'_3|$. Then,
$\sum_{i=1}^p (x_i + y_i) = a$ and $x_j + y_j \le b$ for $1 \le j \le k$, by the definition of $V'_2, V'_3$. Applying Lemma~\ref{lem::pro}, we have
$$
d^{\mH''}_{V'}(v_0) = \sum_{i=1}^p x_iy_i
\le \left\lfloor \frac{a}{b} \right \rfloor \cdot \frac{b^2}{4} + \frac{1}{4}\left(a - b\left \lfloor \frac{a}{b} \right \rfloor \right)^2.
$$
If $b \le (m-1)/3$, then
\begin{align*}
d^{\mH''}_{V'}(v_0)
\le \frac{1}{4}a b + \frac{1}{4} b^2
\le \frac{m-1}{4} \cdot \frac{m-1}{3} + \frac{1}{4} \left(\frac{m-1}{3}\right)^2
= \frac{1}{9}(m-1)^2 < \left(\frac{1}{8} - 3 \beta \right)m^2.
\end{align*}
If $(m-1)/3 < b < (1/2 - 8\beta)m$, then $\lfloor a/b \rfloor =2$ and
\begin{align*}
d^{\mH''}_{V'}(v_0)
&\le 2 \cdot \frac{b^2}{4} + \frac{1}{4} \left( m-1 - 2b \right)^2
= \frac{3}{2} b^2 - (m-1)b + \frac{1}{4} (m-1)^2\\
&\le \frac{3}{2} \left(\left(\frac{1}{2} - 8\beta \right) m \right)^2 -(m-1)\cdot \left(\frac{1}{2} - 8\beta \right) m + \frac{1}{4}(m-1)^2 \\
&
= \left(\frac{1}{8} - 4\beta + 96\beta^2\right)m^2 - 8\beta m + \frac{1}{4}
< \left(\frac{1}{8} - 3 \beta \right)m^2.
\end{align*}
Both cases contradict~\eqref{equ::dmHppVpv0}.
Therefore, $|V'_2|+|V'_3| \ge (1/2 - 8\beta) m$.
\end{proof}

\begin{claim} \label{cla::V'1V'2V'3LowerBound}
We have $|V'_1|, |V'_2|, |V'_3| \ge 0.02m$.
\end{claim}
\begin{proof}
Assume for contradiction that $|V'_1| < 0.1m$ and hence $|V'_2| + |V'_3| > 0.9m$.
If $|V'_2| \ge |V'_3|$, then $|V'_2| \ge 0.45m$.
By~\eqref{equ::Tppdu}, we have $|V'_3| > 0.9m - (1/2 + \beta)m = (0.4 - \beta)m$.
 By Lemma~\ref{lem::degWithin}, there exists a vertex $v \in V'_3$ such that $|N^-(v) \cap V'_3| \ge |V'_3| / 4$. We have
$$
d^-(v) \ge |V'_2| +  |N^-(v) \cap V'_3|
\ge |V'_2| + \frac{|V'_3|}{4}
> 0.45m + \left(0.1 - \frac{\beta}{4} \right) m
> \left( \frac{1}{2} + \beta\right) m
,
$$
a contradiction to~\eqref{equ::Tppdu}. A similar argument holds if $|V'_2| < |V'_3|$ by considering the out-degree of some vertex $v \in V_2'$. Therefore, $|V'_1| \ge 0.1 m$.

Assume for contradiction that $|V_2'| < 0.02m$. Then, by~\eqref{equ::Tppdl}, we have $|V'_{1a}| \ge (0.48 -\beta) m$, and by Claim~\ref{cla::sizeV'2V'3}, we have $|V'_3| \ge (0.48-8\beta)m$. By Lemma~\ref{lem::degWithin}, there exists $v \in V'_{1a}$ such that $|N^-(v) \cap V'_{1a}| \ge (0.12 - \beta / 4) m$. Then, we $d^-(v) \ge (0.6- 9 \beta)m$, a contradiction to~\eqref{equ::Tppdu}. Therefore, $|V'_2| \ge 0.02m$. Similar arguments give $|V'_3| \ge 0.02m$. \qedhere
\end{proof}

We define
\begin{alignat*}{2}
E_{12} &\ce \{\,\{u,v\}: u \in V'_{1a}&&,\,v \in V'_2,\,u \gets v \textrm{ in $T''$}\,\}, \\
E_{13} &\ce \{\,\{u,v\}: u \in V'_{1b}&&,\,v \in V'_3,\,u \to v \textrm{ in $T''$}\,\}.
\end{alignat*}

\begin{claim} \label{cla::sizeE12E13}
We have
$|E_{12}|, |E_{13}| \le 100\beta m^2$.
\end{claim}
\begin{proof}
Let $S_1 \ce \{ u \in V'_{1a}: \{u\} \gets V'_{1b}\}$. For every $u \in S_1$, if $|N^-(u) \cap V'_2| > 2\beta m$, then using~\eqref{equ::Tppdl}, we have
$$
d^-(u) \ge |N^-(u) \cap V'_2| + |V'_3| + |V'_{1b}|
 > 2\beta m + d^-(v_0)
>\left(\frac{1}{2} + \beta \right) m,
$$
a contradiction to~\eqref{equ::Tppdu}. Hence $|N^-(u) \cap V'_2| \le 2 \beta m$ for every $ u \in S_1$.

Let $S_2 \ce \{u \in V'_{1a} \sm S_1: \textrm{there exists $x \in V'_2$ such that } x \to u\}$. We claim $|S_2| \le 60 \beta m$. Otherwise, by Lemma~\ref{lem::degWithin}, there exists $u \in S_2$ with $|N^-(u) \cap S_2| \ge 15\beta m$. By the definition of $S_2$, there exist $x \in V'_2$ with $x \to u$ and $z \in V'_{1b}$ with $u \to z$. Let $y$ be an arbitrary vertex in $V'_3$, which exists since $V'_3 \neq \es$ by Claim~\ref{cla::V'1V'2V'3LowerBound}. Now, we have $v_0 \to \{x,u\}$, $v_0 \gets \{y,z\}$, $x \to \{y,u\}$, $x \gets z$, $u \gets y$, and $u \to z$. If $y \to z$, then $v_0uzxy$ forms a copy of $\Cm[5]$ in $\mH''$, a contradiction. Hence, $y \gets z$. For every $x' \in V'_2 \sm \{x\}$, if $x' \gets u$, then $v_0x'yuz$ forms a copy of $\Cm[5]$ in $\mH''$, a contradiction, so we have $V'_2 \to \{u\}$. Then, using Claim~\ref{cla::sizeV'2V'3}, we have
$$
d^-(u) \ge |N^-(u) \cap S_2| + |V'_2| + |V'_3|
\ge 15 \beta m + \left( \frac{1}{2} - 8 \beta \right)m > \left( \frac{1}{2} +  \beta \right)m,
$$
a contradiction to~\eqref{equ::Tppdu}. Therefore, $|S_2| \le 60 \beta m$. Thus,
$$
|E_{12}| \le |S_1| \cdot 2\beta m  + |S_2| \cdot |V'_2|
\le m \cdot 2\beta m + 60 \beta m \cdot m \le 100\beta m^2.
$$
Similar arguments give the upper bound on $|E_{13}|$.
\end{proof}

\begin{claim} \label{cla::uppbadEdgesp}
We have $|\mH^{''\pi'}_{bad}| \le 200\beta m^3$.
\end{claim}
\begin{proof}
Recall that we have $V'_2 \to V'_3$. For vertices $x \in V'_1, y \in V'_2$, we have $x \to y$ unless $\{x,y\} \in E_{12}$. For vertices $x \in V'_1, z \in V'_3$, we have $x \gets z$ unless $\{x,z\} \in E_{13}$. Therefore, every hyperedge in $\mH^{''\pi'}_{bad}$ contains a pair in $E_{12} \cup E_{13}$, so $|\mH^{''\pi'}_{bad}| \le 200\beta m^3$ by Claim~\ref{cla::sizeE12E13}.
\end{proof}

\begin{claim} \label{cla::sizeVp}
We have $(1/3 - \eps_2/2)m < |V'_1|,|V'_2|,|V'_3| < (1/3 + \eps_2/4) m$.
\end{claim}
\begin{proof}

By Claim~\ref{cla::V'1V'2V'3LowerBound}, we can assume that $0.02m \le |V'_i| \le |V'_j| \le |V'_k|$, where $\{i,j,k\} = \{1,2,3\}$. Assume for contradiction that $|V'_k| \ge (1/3 + \eps_2/4)m$. Then, $|V'_i| \le (m - |V'_k|)/2 \le (1/3 - \eps_2/8)m$.
If $j \equiv i +1 \pmod{3}$, we consider the in-degree of vertices in $V'_j$. Let
$$
d^-(V'_j) \ce |\{(u,v): u\in V', v \in V'_j, u \to v \}|.
$$
By~\eqref{equ::Tppdl}, we have $d^-(V'_j) > |V'_j|(1/2 - \beta)m$. On the other hand, by the definition of $E_{12},E_{13}$ and Claim~\ref{cla::sizeE12E13}, we have
\begin{multline*}
d^-(V'_j) \le |V'_i||V'_j| + \b{|V'_j|}{2} + |E_{12}| + |E_{13}| \le \left( |V'_i| + \frac{|V'_j|}{2} \right)|V'_j| + 200 \beta m^2 \\
=
\left(\frac{m-|V'_k|}{2} + \frac{|V'_i|}{2} \right)|V'_j| + 200\beta m^2
\le |V'_j| \left(\frac{1}{2} - \frac{3\eps_2}{16}\right) m + 200 \beta m^2 <|V'_j|\left( \frac{1}{2} - \beta \right) m,
\end{multline*}
a contradiction, where the last inequality is because $|V_j'| \ge 0.02m$ and $\eps_2 \gg \beta$. If $j \equiv i - 1 \pmod{3}$, a similar argument holds by considering the out-degree of vertices in $V'_j$. Thus, we get $|V'_k| < (1/3 + \eps_2/4)m$ and hence $|V'_i| \ge m - |V'_j| - |V'_k| > (1/3 - \eps_2/2)m$.
\end{proof}

Finally, let $V_1 \ce V'_1 \cup \{V(\mH) \sm V'\}$, $V_2 \ce V_2'$, and $V_3 \ce V_3'$.
Consider the partition $\pi \ce (V_1,V_2,V_3)$ of $V(\mH)$.
Recall that $|V'| = m \ge (1-\beta /2)n$. By Claim~\ref{cla::sizeVp}, we have
$$
\left(\frac{1}{3} - \eps_2 \right)n < |V_1|,|V_2|,|V_3| < \left( \frac{1}{3} + \eps_2 \right)n.
$$
Recall that $T''$ is a subtournament of $T'$ induced by $V'$ and $\mH' = \mH(T')$. For every hyperedge $e$ in $\mH^{'\pi}_{bad}$, either $e \cap (V(\mH) \sm V') \neq \es$ or $e \in \mH^{''\pi'}_{bad}$. Hence, by Claim~\ref{cla::uppbadEdgesp}, we have
$$
|\mH^{'\pi}_{bad}|
\le |V(\mH) \sm V'| \cdot n^2 +  |\mH^{''\pi'}_{bad}|
\le \frac{\beta}{2} n^3 + 200 \beta m^3
\le \frac{\beta}{2} n^3 + 200 \beta n^3
\le 300 \beta n^3.
$$
Recall that $T$ and $T'$ differ by at most $\gamma n^2/2$ edges, so $\mH(T)$ and $\mH'$ differ by at most $\gamma n^3/2$ hyperedges. Also recall that $\mH \subseteq \mH(T)$. Thus,
\begin{equation*}
\left| \mH^\pi_{bad} \right|
\le |\mH(T)^\pi_{bad}|
\le |\mH^{'\pi}_{bad}| + \frac{\gamma}{2} n^3
\le 300 \beta n^3 + \frac{\gamma}{2} n^3
< \eps_1 n^3. \qedhere
\end{equation*}
\end{proof}

\section{Proof of Theorem~\ref{thm::main}} \label{sec::Pro}
In this section, we provide the proof for Theorem~\ref{thm::main}.
Recall that we defined in Section~\ref{sec::Cyc} that $\fCm$ is the set of all the pseudo-cycles minus one hyperedge of size $\l$, where $4 \le \l \le \L$ and $3 \nmid \l$, and $\ifCm = \bigcup_{\L \ge 4} \fCm$.
We will first prove that $\pi(\fCm) = 1/4$ for sufficiently large $\L$, see Proposition~\ref{pro::main}, and then convert it to Theorem~\ref{thm::main} using Theorem~\ref{thm::turanBlowup}.

If we could prove that for some $\L \ge 5$, the maximum $\fCm$-free $3$-graphs can be made free of $\ifCm$ by removing $o(n^3)$ hyperedges, then we get $\pi(\fCm) \le 1/4$ immediately by Lemmas~\ref{lem::numCycTri} and~\ref{lem::freifCmOri}. We are not able to show this. The following lemma is what we can achieve. Its proof is parallel to Propositions 6.3 and 6.4 in~\cite{kamvcev2022turan}.

\begin{lemma} \label{lem::rem}
If a $3$-graph $\mH$ is free of $\fCm$, where $\L \ge 27$, then we can delete at most $\frac{1}{2}\sqrt{\frac{21}{\L-26}}n^3$ hyperedges from $\mH$ to make it free of $\ifCm$.
\end{lemma}
\begin{proof}
Let $\delta = \sqrt{\frac{21}{\L-26}}$. Do the following algorithm for $\mH$. Check whether there is a pair of vertices $\{u,v\}$ with codegree in $(0,\delta n)$. If so, delete all the hyperedges containing $\{u,v\}$ and check again; otherwise end the algorithm. Call the remaining hypergraph $\mHp$. In $\mHp$, every pair of vertices has codegree either $0$ or at least $\delta n$. Note that $\mHp$ is $\fCm$-free and $|\mHp| \ge |\mH| -\delta n^3/2$, since for every pair of vertices $\{u,v\}$, the algorithm does the removal at most once.
We will prove that $\mHp$ is free of $\ifCm$.

We first show that whenever $\mHp$ contains a pseudo-path $abu_1\ldots u_t cd$, for some $t \ge 0$, then $\mHp$ also contains a pseudo-path $abv_1\ldots v_k cd$, where $k < (\L-8)/3$. Let $\mP = abv_1\ldots v_k cd$ be the shortest pseudo-path starting from $a, b$ and ending with $c,d$.
For vertices $v,v'$, let $M(v,v') \ce \{(u,u') : \{v,v',u\}, \{v',u,u'\} \in \mHp\}$.
For $1 \le i < j \le k$, where both $i,j$ are divisible by $7$, we claim that $M(v_i, v_{i+1}) \cap M(v_j, v_{j+1}) = \es$. Otherwise, let $(w,w') \in M(v_i, v_{i+1}) \cap M(v_j, v_{j+1})$. Then,
$$abv_1 \ldots v_{i-1}v_iv_{i+1}ww'v_{j+1}wv_jv_{j+1}v_{j+2}\ldots v_k cd$$
is a shorter pseudo-path between $a,b$ and $c,d$ than $\mP$, a contradiction. By the codegree condition of $\mHp$, we have $|M(v_i,v_{i+1})| \ge \delta^2n^2$, for every $1\le i < k$. Therefore,
$$
n^2
> \left|\bigcup_{i\,:\,1\le i< k} M(v_i,v_{i+1})\right|
\ge \sum_{i\,:\,1\le i< k,\,7\,\mid\,i} \left|M(v_i,v_{i+1})\right|
\ge
\left\lfloor \frac{k}{7} \right\rfloor \cdot \delta^2 n^2
\ge \left(\frac{k}{7} - \frac{6}{7}\right) \cdot \delta^2 n^2,
$$
implying $k < 7/\delta^2 + 6 = (\L-8)/3$, as desired.

Now, assume for contradiction that $u_1u_2 \ldots u_{\l - 1} u_\l$ is the shortest pseudo-cycle minus one hyperedge in $\mHp$ with size $\l$ not divisible by $3$. Note that $\l > \L$.
Let $s = \lfloor (\L+1)/3 \rfloor + 1$. Since there is a pseudo-path $u_1u_2\ldots u_s u_{s+1}$ in $\mHp$, by the claim in the last paragraph, there is a pseudo-path $\mP = u_1u_2v_1v_2 \ldots v_k u_s u_{s+1}$, where $k < (L-8)/3$. Then, $\mP u_{s+2}\ldots u_{\l - 1} u_\l$ is a pseudo-cycle minus one hyperedge of size $k+4 + (\l -s -1) < \l$ in $\mHp$,
so $3\mid k+4 + (\l - s-1)$ and hence $3 \mid k+2s+\l$. Let $\mP'$ be the pseudo-path $u_s u_{s+1} u_{s-1} u_s u_{s-2} u_{s-1} \ldots u_4u_2u_3u_1u_2$. Then, $v_1 v_2 \ldots v_k \mP'$ is a pseudo-cycle of size $k+2s < \L$ in $\mHp$, so $3\mid k+2s$. However, we get $3 \mid \l$, a contradiction. Thus, $\mHp$ is free of $\ifCm$.
\end{proof}

\begin{lemma} \label{lem::weakTuranNum}
For every integer $\L \ge 27$, we have
$$
\ex(n, \fCm) \le \left(\frac{1}{24} + \frac{1}{2} \sqrt{\frac{21}{\L-26}} \right)n^3.
$$
\end{lemma}
\begin{proof}
This follows immediately from Lemmas~\ref{lem::numCycTri},~\ref{lem::freifCmOri}, and~\ref{lem::rem}.
\end{proof}

We use a standard symmetrization method in the following lemmas to bound the degrees in a maximum $\fCm$-free $3$-graph.
For a $3$-graph $\mH$, a vertex set $S \subseteq V(\mH)$, and a vertex $v \in V(\mH)\sm S$, let $T_{S,v} = \{v_1,\ldots, v_{|S|}\}$ be a set such that $T_{S,v} \cap V(\mH) = \es$, and let $\mH_{S,v}$ be the $3$-graph on vertex set $(V(\mH) \sm S) \cup T_{S,v}$ with hyperedges
$$\left(\mH \sm \{e \in \mH: e \cap S \neq \es\}\right) \bigcup
\left( \bigcup_{i=1}^{|S|} \{\{v_i, x, y\}: \{v, x, y\} \in \mH, \{x,y\} \cap S = \es\}  \right)
.
$$
We write $\mH_{u,v}$ for $\mH_{\{u\},v}$.

\begin{lemma} \label{lem::sym}
If a $3$-graph $\mH$ is $\fCm$-free for some $\L \ge 4$, then $\mH_{S,v}$ is also $\fCm$-free, for every vertex set $S \subseteq V(\mH)$ and vertex $v \in V(\mH) \sm S$.
\end{lemma}
\begin{proof}
Assume for contradiction that $\mH_{S,v}$ is not $\fCm$-free. Let $v_1v_2 \ldots v_\l$ be a copy of some $\C^- \in \fCm$ in $\mH_{S,v}$.
Note that the codegree of $v,u'$ and the codegree of $u',u''$ are zero in $\mH_{S,v}$, for every $u' \neq u'' \in T_{S,v}$. Replacing all appearances of the new vertices in $T_{S,v}$ with $v$ in $v_1v_2\ldots v_\l$, we get a copy of $\C^- \in \fCm$ in $\mH$, a contradiction to that $\mH$ is $\fCm$-free.
\end{proof}

\begin{lemma} \label{lem::degl}
There exists a constant $N> 0$ such that the following is true for every $\L \ge 4$ and $n > N$. Let $\mH$ be a maximum $n$-vertex $\fCm$-free $3$-graph. For every vertex $v \in V(\mH)$, we have $d(v) \ge n^2/8 - 2n$.
\end{lemma}
\begin{proof}
Assume for contradiction that there exists $v \in V(H)$ with $d(v) < n^2/8 - 2n$. By Construction~\ref{con::iteEdge}, $|\mH| \ge n^3/24 - Cn \log n$ for some constant $C> 0$. Hence, there is $x \in V(\mH)$ such that $d(x) \ge n^2 / 8 - 3C \log n$. By Lemma~\ref{lem::sym}, $\mH_{v,x}$ is also $\fCm$-free.
However, we have
$$
|\mH_{v,x}| > |\mH| - (n^2/8 -2n) + (n^2/8 - 3C \log n) - n > |\mH|,
$$
a contradiction to the maximality of $\mH$.
\end{proof}

\begin{lemma} \label{lem::degu}
For every $\eps > 0$, there exist $\delta \ge \eps^2/100$, $\L_0$, and $N$ such that the following is true for every $\L > \L_0$ and $n > N$. Let $\mH$ be an $n$-vertex $\fCm$-free $3$-graph with $|\mH| \ge (1/24 - \delta) n^3$. For every vertex $v \in V(\mH)$, we have $d(v) \le (1/8+\eps)n^2$.
\end{lemma}
\begin{proof}
Without loss of generality, we can assume $\eps$ is sufficiently small.
Let $\delta = \eps^2/100$, and let $\L_0 > 21\cdot 50^2\cdot (1/\eps)^4 + 26$ and $N$ be sufficiently large. Let $\L > \L_0$.
Note that for every $\fCm$-free $3$-graph $\F$, by Lemma~\ref{lem::weakTuranNum}, we have $|\F| \le (1/24 + \delta)n^3$.

Let $\mH$ be an $n$-vertex $\fCm$-free $3$-graph with $|\mH| \ge (1/24 - \delta) n^3$.
Assume for contradiction that a vertex $v\in V(\mH)$ has $d(v) > (1/8+\eps)n^2$.
Let $S_0 = \{u \in V(\mH): d(u) \le (1/8+3\sqrt{\delta}) n^2\}$. We have
$$
|V(\mH) \sm S_0| \cdot \left(\frac{1}{8} + 3\sqrt{\delta} \right) n^2 \cdot \frac{1}{3} \le |\mH|  < \left( \frac{1}{24} + \delta \right)n^3,
$$
so
$$
|S_0| > \frac{\sqrt{\delta} - \delta}{\frac{1}{24} + \sqrt{\delta}} n > \sqrt{\delta} n.
$$
Let $S$ be a subset of $S_0$ with size $\lfloor\sqrt{\delta} n \rfloor$. Note that $v \notin S_0 \supseteq S$.
By Lemma~\ref{lem::sym}, $\mH_{S,v}$ is $\fCm$-free. However, we have
\begin{align*}
|\mH_{S,v}| &> |\mH| - |S|\cdot \left(\frac{1}{8} + 3\sqrt{\delta}\right) n^2 + |S| \cdot \left(\left( \frac{1}{8} + \eps \right)n^2 - |S|\cdot n \right) \\
&\ge \left(\frac{1}{24} - \delta \right) n^3 + |S|\left(\eps - 4\sqrt{\delta}\right)n^2
 \ge \left(\frac{1}{24} - \delta \right) n^3 + \left(\sqrt{\delta} n - 1\right)\left(\eps - 4\sqrt{\delta}\right)n^2
 \\
& \ge \left(\frac{1}{24} - \delta + \sqrt{\delta}\left(\eps - 4\sqrt{\delta}\right)\right) n^3 - \eps n^2
> \left( \frac{1}{24} + \delta \right)n^3
,
\end{align*}
a contradiction. Therefore, $d(v) \le (1/8+\eps)n^2$ for every vertex $v \in V(\mH)$.
\end{proof}

\begin{proposition} \label{pro::main}
For every sufficiently large $\L$, we have $\pi(\fCm) = 1/4$.
\end{proposition}

\begin{proof}
For every integer $\L \ge 4$, by Construction~\ref{con::iteEdge}, we have $\ex(n,\fCm) \ge n^3/24 -C n \log n$ for some absolute constant $C >0$.

For the upper bound, fix $\eps$ to be $10^{-10000}$ and let $\L,M$ be sufficiently large integers. We will prove that for every $n \ge 1$, we have
\begin{equation} \label{equ::ex}
\ex(n,\fCm) \le n^3/24 + M^2n.
\end{equation}
When $n \le M$,~\eqref{equ::ex} is trivial, since $M^2n > \b{n}{3}$. Now assume that $n > M$ and~\eqref{equ::ex} is true for every positive integer less than $n$.
Recall that for a $3$-graph $\mH$ and a partition $\pi = (V_1,V_2,V_3)$ of $V(\mH)$, we defined $\mH_\pi \ce \mH[V_1,V_2,V_3]$ and $\bar{\mH}_\pi \ce [V_1,V_2,V_3] \sm \mH_\pi$ in Section~\ref{sec::Pre} and $\mH^\pi_{bad} \ce \bigcup_{1\le i \neq j \le 3} \mH[V_i, V_i, V_j]$ in Section~\ref{sec::Sta}.

Let $\mH$ be a maximum $n$-vertex $\fCm$-free $3$-graph.
By Construction~\ref{con::iteEdge}, $|\mH| \ge n^3/24 -C n \log n$.
By Lemmas~\ref{lem::degl} and~\ref{lem::degu}, we have for every vertex $v \in V(\mH)$,
\begin{equation} \label{equ::deg}
\left(1/8 - \eps \right) n^2 < d(v) < \left(1/8 + \eps \right) n^2.
\end{equation}
By Lemma~\ref{lem::rem}, there exists a $\ifCm$-free subhypergraph $\mHp \subseteq \mH$ with $|\mH \sm \mHp| \le \frac{1}{2}\sqrt{\frac{21}{\L-26}} n^3$, so $|\mHp| \ge \left(\frac{1}{24} - \frac{1}{2}\sqrt{\frac{21}{\L-26}}\right)n^3 -C n \log n$.
By Lemma~\ref{lem::freifCmOri}, $\mHp$ is orientable.
Hence, we can apply Proposition~\ref{pro::sta} to $\mHp$ and get a $3$-partition $\pi = (V_1,V_2,V_3)$ of $V(\mHp) = V(\mH)$ such that
\begin{equation} \label{equ::sizeV}
\left(1/3 - \eps\right)n < |V_1|,|V_2|,|V_3| < \left(1/3 + \eps\right)n
\end{equation}
and $|\mH^{'\pi}_{bad}| < \eps n^3/2$. Then, $|\mH^\pi_{bad}| < \eps n^3/2 + \frac{1}{2}\sqrt{\frac{21}{\L-26}} n^3 < \eps n^3$.
Note that $(\mH \sm \mH^\pi_{bad}) \cup \bar{\mH}_\pi$ is still $\fCm$-free.
By the maximality of $\mH$, we have
\begin{equation} \label{equ::sizeBarBad}
|\bar{\mH}_\pi| \le |\mH^\pi_{bad}| < \eps n^3.
\end{equation}
For $\{i,j,k\} = \{1,2,3\}$, we define
\begin{alignat*}{2}
&A_i &&\ce \{x \in V_i\;|\; d_{V_j,V_k}(x) \ge |V_j||V_k| - \eps^{1/2} n^2\}, \quad B_i \ce V_i \sm A_i,\\
&A   &&\ce A_1 \cup A_2 \cup A_3, \quad \textrm{and} \quad B \ce B_1 \cup B_2 \cup B_3.
\end{alignat*}

\begin{claim} \label{cla::sizeB}
We have $|B_i| \le \eps^{1/2} n$ for $i \in [3]$.
\end{claim}
\begin{proof}
If $|B_i|> \eps^{1/2} n$ for some $i \in [3]$, then $|\bar{\mH}_\pi| \ge |B_i| \cdot \eps^{1/2} n^2 >  \eps n^3$, contradicting~\eqref{equ::sizeBarBad}.
\end{proof}

\begin{claim} \label{cla::OneLargeOneSmall}
Let $\{i,j,k\} = [3]$.
For vertices $v_1, v_2 \in V(\mH)$ with $d_{V_j,V_k}(v_1) \ge 2\eps^{1/4}n^2$, we have
\begin{equation} \label{equ::mainCleaningClaim}
d_{V_i}(v_1,v_2)
\le
\frac{3n}{d_{V_j,V_k}(v_1)} \left( |V_i||V_k| - d_{V_i,V_k}(v_2)  \right) + 2 \eps^{1/2}n
.
\end{equation}
\end{claim}
\begin{proof}
Assume for contradiction that~\eqref{equ::mainCleaningClaim} is false for some $v_1,v_2 \in V(\mH)$. We will prove that $\mH$ contains a copy of $\Cm[5]$.
Let $Z = \{z \in V_k: d_{V_j}(v_1,z) \ge d_{V_j,V_k}(v_1) / n\}$. Then, by the upper bounds on $|V_j|,|V_k|$ in~\eqref{equ::sizeV}, we have
$$
d_{V_j,V_k}(v_1) \le |Z| \cdot |V_j| + |V_k \sm Z| \cdot \frac{d_{V_j,V_k}(v_1)}{n} \le |Z| \cdot \left(\frac{1}{3} + \eps\right) n + \left(\frac{1}{3} + \eps \right) d_{V_j,V_k}(v_1),
$$
so
\begin{equation} \label{equ::lowerBoundZ}
|Z|
\ge \frac{\frac{2}{3}-\eps}{\left(\frac{1}{3} + \eps \right) n} \cdot d_{V_j,V_k}(v_1)
\ge \frac{3}{2n}d_{V_j,V_k}(v_1).
\end{equation}
Let $X = \{x \in N_{V_i}(v_1,v_2): d_{Z}(x,v_2) \ge |Z|/2 \}$. If $|X| \le d_{V_i}(v_1,v_2)/2$, then by~\eqref{equ::lowerBoundZ} and our assumptions that~\eqref{equ::mainCleaningClaim} is false and $d_{V_j,V_k}(v_1) \ge 2\eps^{1/4}n^2$, we have
\begin{align*}
|V_i||V_k| - d_{V_i,V_k}(v_2)
&\ge |N_{V_i}(v_1,v_2) \sm X| \cdot \left( |Z| - \frac{|Z|}{2} \right)
\ge \frac{1}{2} d_{V_i}(v_1,v_2) \cdot \frac{3}{4n}d_{V_j,V_k}(v_1) \\
&\ge \frac{3}{8n} d_{V_j,V_k}(v_1) \left(\frac{3n}{d_{V_j,V_k}(v_1)} \left( |V_i||V_k| - d_{V_i,V_k}(v_2)  \right) + 2 \eps^{1/2}n\right) \\
d_{V_j,V_k}(v_1)
&\ge \frac{9}{8}\left(|V_i||V_k| - d_{V_i,V_k}(v_2)\right) + \frac{3}{2}\eps^{3/4}n^2
> |V_i||V_k| - d_{V_i,V_k}(v_2)
,
\end{align*}
a contradiction, so
$
|X| > d_{V_i}(v_1,v_2)/2 \ge \eps^{1/2}n^2.
$
By Claim~\ref{cla::sizeB}, $|B_i| \le \eps^{1/2}n$. Hence, we are able to choose and fix a vertex $x \in X \sm B_i \subseteq A_i$. We have $\{v_1,v_2,x\} \in \mH$.

Let $Z' = N_Z(x,v_2)$. By the definition of $X$, we have $|Z'| \ge |Z| /2$. Then, by the definition of $Z$ and~\eqref{equ::lowerBoundZ}, we have
$$
d_{V_j, Z'}(v_1) \ge \frac{d_{V_j,V_k}(v_1)}{n} \cdot |Z'| \ge \frac{d_{V_j,V_k}(v_1)}{n} \cdot \frac{|Z|}{2} \ge \frac{3}{4n^2} d^2_{V_j,V_k}(v_1)
$$
and hence, by the assumption that $d_{V_j,V_k}(v_1) \ge 2\eps^{1/4}n^2$, we have
\begin{equation} \label{equ::dVjsmv_2Z'v1}
d_{V_j\sm\{v_2\}, Z'}(v_1) \ge \frac{3}{4n^2} d^2_{V_j,V_k}(v_1) - |Z'|
\ge \frac{3}{4n^2} \cdot 4 \eps^{1/2} n^4 - n
= 3\eps^{1/2}n^2 - n
> 2\eps^{1/2} n^2.
\end{equation}
We can choose and fix $ y \in V_j \sm \{v_2\}$ and $z \in Z'$ such that $\{y,z\} \in N(v_1) \cap N(x)$, since otherwise, by~\eqref{equ::dVjsmv_2Z'v1}, we have
$$
d_{V_j,V_k}(x) \le |V_j||V_k| - d_{V_j\sm\{v_2\}, Z'}(v_1)
< |V_j||V_k| - 2\eps^{1/2} n^2
,
$$
a contradiction to that $x \in A_i$. Note that $y \neq v_2$. We have $\{x,v_2,z\} \in \mH$ by the definition of $Z'$. We also have $\{v_1,y,z\}, \{x,y,z\} \in \mH$.

Now, we have hyperedges $\{v_1,v_2,x\}$, $\{x,v_2,z\}$, $\{x,y,z\}$, and $\{v_1,y,z\}$, so $v_1v_2xzy$ is a copy of $\Cm[5]$ in $\mH$, a contradiction. Thus, we have~\eqref{equ::mainCleaningClaim}.
\end{proof}

\begin{claim} \label{cla::noij}
Let $\{i,j,k\} = [3]$.
For every vertex $v \in V(\mH)$ with $d_{V_j,V_k}(v) \ge 2\eps^{1/4}n^2$, we have $d_{V_i,V_j}(v) \le \eps^{1/4} n^2$.
\end{claim}
\begin{proof}
For every vertex $y \in A_j$, by Claim~\ref{cla::OneLargeOneSmall}, we have
$$
d_{V_i} (v,y) \le \frac{3n}{2\eps^{1/4}n^2} \cdot \eps^{1/2} n^2 + 2\eps^{1/2} n \le 2 \eps^{1/4} n.
$$
Therefore, by~\eqref{equ::sizeV} and Claim~\ref{cla::sizeB}, we have
\begin{equation*}
d_{V_i,V_j}(v) \le |A_j| \cdot 2\eps^{1/4} n + |B_j| \cdot |V_i| \le \left(\frac{1}{3} + \eps \right)n \cdot \left(2\eps^{1/4}n + \eps^{1/2} n\right) \le  \eps^{1/4} n^2. \qedhere
\end{equation*}
\end{proof}

\begin{claim} \label{cla::mHVi}
We have $|\mH[V_i]| \ge (1/24 - 27\eps^{1/4}) (n/3)^3$ for $i \in [3]$.
\end{claim}
\begin{proof}
Let $\{i,j,k\} = [3]$.
Let $v$ be an arbitrary vertex in $A_i$.
By the definition of $A_i$ and Claim~\ref{cla::noij}, we have $d_{V_i,V_j}(v), d_{V_i,V_k}(v) \le \eps^{1/4} n^2$.
Since $\K_4^- \in \fCm$, the link graph of $v$ is triangle-free. By Theorem~\ref{thm::Mantel} and~\eqref{equ::sizeV}, we have
$$
d_{V_j\cup V_k, V_j \cup V_k}(v) \le \frac{1}{4}\left(|V_j|+|V_k|\right)^2
\le \frac{1}{4}\left( 2 \cdot \left(\frac{1}{3} + \eps\right)n \right)^2
\le \left(\frac{1}{9} + \eps\right)n^2
.
$$
Then, by the bound $d(v) \ge (1/8 -\eps)n^2$ in~\eqref{equ::deg}, we have
$$
d_{V_i,V_i}(v) = d(v) - d_{V_i,V_j}(v)- d_{V_i,V_k}(v) - d_{V_j\cup V_k, V_j \cup V_k}(v)
\ge \left(\frac{1}{72} - 2\eps^{1/4} - 2\eps \right) n^2.
$$
Thus, by~\eqref{equ::sizeV} and Claim~\ref{cla::sizeB}, we have
\begin{align*}
|\mH[V_i]| &\ge \frac{1}{3} \sum_{v \in A_i} d_{V_i,V_i}(v)
\ge \frac{1}{3} \cdot (|V_i| - |B_i|) \cdot  \left(\frac{1}{72} - 2\eps^{1/4}-2\eps\right) n^2
\\
&\ge \frac{1}{3} \cdot \left( \frac{1}{3} -\eps - \eps^{1/2} \right)n \cdot \left(\frac{1}{72} - 3\eps^{1/4}\right) n^2
\ge \left(\frac{1}{24} - 27\eps^{1/4}\right) \left(\frac{n}{3}\right)^3.
\qedhere
\end{align*}

\end{proof}

\begin{claim}
For every vertex $v \in B$, there exists a unique pair $\{j,k\} \subset [3]$ such that $d_{V_j,V_k}(v) \ge |V_j||V_k| - 5\eps^{1/9} n^2$.
\end{claim}
\begin{proof}
For every $i \in [3]$, by Claim~\ref{cla::mHVi} and~\eqref{equ::sizeV}, we have
\begin{align*}
|\mH[V_i \cup \{v\}] |
&\ge |\mH[V_i]|
\ge \left(\frac{1}{24} - 27\eps^{1/4}\right) \left(\frac{n}{3}\right)^3 \\
&\ge \left(\frac{1}{24} - 36\eps^{1/4}\right) \left(\left(\frac{1}{3} + \eps \right)n + 1 \right)^3
\ge  \left(\frac{1}{24} - 36\eps^{1/4}\right) |V_i \cup \{v\}|^3.
\end{align*}
Note that, by~\eqref{equ::sizeV}, we have $|V_i \cup \{v\}| \ge |V_i| \ge (1/3 - \eps)n$, so $|V_i \cup \{v\}|$ is sufficiently large. We can apply Lemma~\ref{lem::degu} to $\mH[V_i \cup \{v\}]$ and get
\begin{equation} \label{equ::dViVi}
d_{V_i, V_i}(v) \le \left( \frac{1}{8} + 60\eps^{1/8}\right) |V_i \cup \{v\}|^2
\le \left( \frac{1}{8} + 60\eps^{1/8} \right) \left( \left( \frac{1}{3} + \eps \right)n + 1 \right)^2
\le \left(\frac{1}{72} + 8\eps^{1/8} \right)n^2.
\end{equation}
If $d_{V_1,V_2}(v), d_{V_1,V_3}(v), d_{V_2,V_3}(v) < 2 \eps^{1/4}n$, then we have
$$
d(v) \le \sum_{i=1}^3 d_{V_i,V_i}(v) + \sum_{1 \le i< j \le 3} d_{V_i,V_j} (v)
\le 3 \cdot \left(\frac{1}{72} + 8\eps^{1/8} \right)n^2 + 3 \cdot 2\eps^{1/4}n^2
< \frac{n^2}{23},
$$
a contradiction to~\eqref{equ::deg}. Hence, we can fix $\{j,k\} \subset [3]$ such that $d_{V_j,V_k}(v) \ge 2 \eps^{1/4} n^2$. Let $\{i\} = [3] \sm \{j,k\}$. By Claim~\ref{cla::noij}, we have
\begin{equation} \label{equ::dViVj}
d_{V_i,V_j}(v), d_{V_i,V_k}(v) \le \eps^{1/4}n^2.
\end{equation}
It remains to prove that $d_{V_j,V_k}(v) \ge |V_j||V_k| - 5 \eps^{1/9} n^2$.

Let $G = N_{V_j \cup V_k, V_j \cup V_k} (v)$. We view $G$ as a graph on vertex set $(V_j \cup V_k) \sm \{v\}$.
Note that $G$ is triangle-free, since $\mH$ is $\K_4^-$-free.
By~\eqref{equ::sizeV}, we have
\begin{equation} \label{equ::numVerVG}
|V(G)| \le |V_j| + |V_k| \le \left(\frac{2}{3}+2\eps\right)n
.
\end{equation}
By~\eqref{equ::deg},~\eqref{equ::dViVi}, and~\eqref{equ::dViVj}, we have
\begin{multline} \label{equ::dVjVjVjVkVkVk}
|G| = d_{V_j,V_j}(v) + d_{V_j,V_k}(v) + d_{V_k,V_k}(v)
= d(v) - d_{V_i,V_i}(x) - d_{V_i,V_j}(x) - d_{V_i,V_k}(x)
\\ \ge
\left( \frac{1}{8} - \eps \right) n^2 - \left(\frac{1}{72} + 8\eps^{1/8} \right)n^2 - 2\cdot \eps^{1/4} n^2
\ge \left(\frac{1}{9} - 9\eps^{1/8}\right) n^2.
\end{multline}

For a bipartition $(S_1,S_2)$ of $V_j \cup V_k$, let $G[S_1,S_2]$ be the set of edges in $G$ between $S_1$ and $S_2$. Fix $(S_1,S_2)$ to be the bipartition of $V_j\cup V_k$ which minimizes $|G\sm G[S_1,S_2]|$. By~\eqref{equ::numVerVG},~\eqref{equ::dVjVjVjVkVkVk}, and Theorem~\ref{thm::manSta}, we have
\begin{multline} \label{equ::GsmGS1S2}
    |G \sm G[S_1,S_2]| \le |G| - \frac{4|G|^2}{|V(G)|^2}
    \le |G| - \frac{4|G|^2}{\left( \frac{4}{9} + \frac{8}{3}\eps + 4\eps^2 \right)n^2} \\
    \le \left(\frac{1}{9} - 9\eps^{1/8}\right) n^2 - \frac{4 \left(\left(\frac{1}{9} - 9\eps^{1/8}\right) n^2\right)^2}{\left( \frac{4}{9} + \frac{8}{3}\eps + 4\eps^2 \right)n^2}
    \le \left(\frac{1}{9} - 9\eps^{1/8}\right) n^2 - \left(\frac{1}{9} -20\eps^{1/8}\right)n^2
    = 11 \eps^{1/8} n^2.
\end{multline}
By~\eqref{equ::dVjVjVjVkVkVk} and~\eqref{equ::GsmGS1S2}, We have $|G[S_1,S_2]| \ge (1/9 - 20 \eps^{1/8})n^2$. Then, by~\eqref{equ::sizeV},
\begin{align} \label{equ::S1S2minusGS1S2}
|[S_1, S_2] \sm G[S_1,S_2]|
&= |S_1||S_2| - |G[S_1,S_2]|
\le \left(\frac{|V_j| + |V_k|}{2} \right)^2 - |G[S_1,S_2]|
\notag \\
&\le
\left(\left(\frac{1}{3} + \eps\right) n \right)^2 -
\left(\frac{1}{9} - 20 \eps^{1/8}\right)n^2
\le 21\eps^{1/8}n^2.
\end{align}
Note that if $\{u,w_1\},\{u,w_2\} \in G$, then $\{u,w_1,w_2\} \notin \mH$, since otherwise, $vw_1uw_2$ forms a copy of $\K_4^-$ in $\mH$. Hence, every hyperedge in $\mH[S_1,S_1,S_2] \cup \mH[S_2,S_2,S_1]$ contains at least one pair of vertices in $[S_1, S_2] \sm G[S_1,S_2]$. Therefore, by~\eqref{equ::S1S2minusGS1S2},
\begin{equation} \label{equ::sizeS1S1S2S2S2S1}
|\mH[S_1,S_1,S_2]| + |\mH[S_2,S_2,S_1]|
\le
21\eps^{1/8}n^3.
\end{equation}

Now, assume for contradiction that $d_{V_j,V_k}(v) < |V_j||V_k| - 5 \eps^{1/9} n^2$. By~\eqref{equ::dVjVjVjVkVkVk} and~\eqref{equ::sizeV},
$$
d_{V_j,V_j}(v) + d_{V_k,V_k}(v) \ge \left(\frac{1}{9} - \eps^{1/8} \right) n^2 - \left( |V_j||V_k| - 5 \eps^{1/9} n^2 \right)
\ge 4 \eps^{1/9} n^2
,
$$
so we can assume without loss of generality that $d_{V_j,V_j}(v) \ge 2\eps^{1/9}n^2$. Let $V_{j1} = V_j \cap S_1$ and $V_{j2} = V_j \cap S_2$. Then, by~\eqref{equ::GsmGS1S2}, we have
\begin{align*}
2\eps^{1/9} n^2 \le
d_{V_j,V_j}(v)
&= d_{V_{j1},V_{j1}}(v) + d_{V_{j2},V_{j2}}(v) + d_{V_{j1},V_{j2}}(v)
\\
&\le |G \sm G[S_1,S_2]| + |V_{j1}||V_{j2}|
\le 11\eps^{1/8}n^2 + |V_{j_1}||V_{j_2}|
,
\end{align*}
so
\begin{equation} \label{equ::sizeVj1Vj2}
|V_{j1}|,|V_{j2}| \ge \eps^{1/9}n.
\end{equation}
By Lemma~\ref{lem::weakTuranNum}, we have
\begin{equation} \label{equ::Vj1Vj2}
|\mH[V_{ja}]|
\le  \left(\frac{1}{24} + \frac{1}{2} \sqrt{\frac{21}{\L-26}} \right) |V_{ja}|^3
\le \left(\frac{1}{24} + \eps \right) |V_{ja}|^3
,
\quad \textrm{for $a = 1,2$}
.
\end{equation}
Then, by~\eqref{equ::Vj1Vj2},~\eqref{equ::sizeS1S1S2S2S2S1},~\eqref{equ::sizeV}, and~\eqref{equ::sizeVj1Vj2},
we have
\begin{align*}
|\mH[V_j]|
&\le |\mH[V_{j1}]| + |\mH[V_{j2}]| + |\mH[S_1,S_1,S_2]| + |\mH[S_2,S_2,S_1]|\\
&\le \left(\frac{1}{24} + \eps\right) \left(|V_{j1}|^3+ |V_{j2}|^3\right) + 21\eps^{1/8}n^3
\le \left(\frac{1}{24} + \eps\right) \left(|V_{j1}|^3+ \left(|V_j|- |V_{j1}|\right)^3\right) + 21\eps^{1/8}n^3\\
&\le \left(\frac{1}{24} + \eps\right) \left(\left(\eps^{1/9}n\right)^3+ \left(\left(\frac{1}{3} + \eps\right)n- \eps^{1/9}n\right)^3\right) + 21\eps^{1/8}n^3
\le \left(\frac{1}{24} - \frac{1}{10} \eps^{1/9}\right) \left(\frac{n}{3}\right)^3,
\end{align*}
a contradiction to Claim~\ref{cla::mHVi},
where the second-to-last inequality is due to the fact that the function $f_c(x) = x^3 + (c-x)^3$, where $c$ is a positive constant, is decreasing on $[0, c/2]$ and increasing on $[c/2,c]$.
Thus, we have $d_{V_j,V_k}(v) \ge |V_j||V_k| - 5 \eps^{1/9} n^2$.
\end{proof}

For $\{i,j,k\} = [3]$,
let $B'_i \ce \{v \in B: d_{V_j,V_k}(v) \ge |V_j||V_k| - 5\eps^{1/9} n^2\}$ and $V_i' \ce A_i \cup B'_i$.
Let $\pi' \ce (V'_1,V'_2,V'_3)$ be a new partition of $V(\mH)$. By~\eqref{equ::sizeV} and Claim~\ref{cla::sizeB}, we have
\begin{equation} \label{equ::sizeV'}
    \left(\frac{1}{3} - 2\eps^{1/2}\right)n < |V'_1|,|V'_2|,|V'_3| < \left( \frac{1}{3} + 3\eps^{1/2}\right)n.
\end{equation}

\begin{claim} \label{cla::allGood}
For $\{i,j,k\} = [3]$ and vertices $x \in V'_i$, $y \in V'_j$, we have $d_{V'_i} (x,y) \le \eps^{1/10}n$.
\end{claim}
\begin{proof}
By Claim~\ref{cla::OneLargeOneSmall} and~\eqref{equ::sizeV}, we have
$$
d_{A_i}(x,y) \le d_{V_i}(x,y) \le \frac{3n}{|V_j||V_k| - 5\eps^{1/9} n^2} \cdot 5\eps^{1/9} n^2 + 2\eps^{1/2} n
\le \frac{1}{2}\eps^{1/10}n
.
$$
By Claim~\ref{cla::sizeB}, we have $d_{B'_i}(x,y)\le |B'_i| \le |B| \le 3\eps^{1/2}n$. Therefore, $d_{V'_i} (x,y) \le \eps^{1/10}n$.
\end{proof}

\begin{claim}
We have $\mH^{\pi'}_{bad} = \es$.
\begin{proof}
For $\{i,j,k\} = \{3\}$, let
$$
Q_{ij} \ce \{\,(\{x_1,x_2,y\},z): \{x_1,x_2,y\} \in \mH[V'_i,V'_i,V'_j],\,x_1,x_2 \in V'_i,\, y\in V'_j,\, z \in V'_k \,\}
$$
and $Q \ce \bigcup_{1\le i \neq j \le 3} Q_{ij}$. Note that for every $(\{x_1,x_2,y\},z) \in Q_{ij}$, we have $\{x_1,y,z\} \in \bar{\mH}_{\pi'}$ or $\{x_2,y,z\} \in \bar{\mH}_{\pi'}$, since otherwise $x_1zyx_2$ forms a copy of $\K_4^-$.
For every $\{x,y,z\} \in
\bar{\mH}_{\pi'}$,
by Claim~\ref{cla::allGood}, there can be at most $\eps^{1/10}n$ such $x' \in V_i'$ that $\{x,x',y\} \in \mH[V'_i,V'_i,V'_j]$, so we have
\begin{equation} \label{equ::QbarmH}
|Q| = \sum_{1 \le i \neq j \le 3} |Q_{ij}| \le 6 \cdot \eps^{1/10} n \cdot |\bar{\mH}_{\pi'}|.
\end{equation}
On the other hand,
fix $i\neq j \in [3]$ which maximize $|\mH[V'_i,V'_i,V'_j]|$. Let $k = [3] \sm \{i,j\}$.
By~\eqref{equ::sizeV'},
\begin{equation} \label{equ::QmHV'iV'iV'j}
|Q| \ge |Q_{ij}| = |\mH[V'_i,V'_i,V'_j]| \cdot |V'_k|
\ge |\mH[V'_i,V'_i,V'_j]| \left(\frac{1}{3} - 2\eps^{1/2}\right)n
.
\end{equation}
Combining~\eqref{equ::QbarmH} and~\eqref{equ::QmHV'iV'iV'j}, we have $|\mH[V'_i,V'_i,V'_j]| \le 0.01|\bar{\mH}_{\pi'}|$, so
$$
|\mH^{\pi'}_{bad}| \le 6 |\mH[V'_i,V'_i,V'_j]| \le \frac{1}{2}|\bar{\mH}_{\pi'}|.
$$
If $|\mH^{\pi'}_{bad}| > 0$, then
$(\mH \sm \mH^{\pi'}_{bad}) \cup \bar{\mH}_{\pi'}$ is $\fCm$-free and has strictly more hyperedges than $\mH$, which is a contradiction to the maximality of $\mH$. Therefore, $\mH^{\pi'}_{bad} = \es$.
\end{proof}
\end{claim}

Note that by~\eqref{equ::sizeV'}, we have $|V_i'| < n$ for $i \in [3]$.
By induction, we now have
\begin{align*}
\ex(n,\fCm) &= |\mH| = |\mH_{\pi'}| + \sum_{i=1}^3 |\mH[V_i']| \le |V_1'||V_2'||V_3'| + \sum_{i=1}^3 \left(\frac{1}{24}|V_i'|^3 + M^2\cdot |V_i'| \right) \\
&= |V_1'||V_2'||V_3'| + \frac{1}{24}\sum_{i=1}^3 |V_i'|^3 + M^2 n \le \frac{1}{24}n^3 + M^2 n
,
\end{align*}
where for the last inequality, we use~\eqref{equ::sizeV'} and the fact that the function $g(x_1,x_2,x_3) = x_1x_2x_3 + \sum_{i = 1}^3 x_i^3 /24$ defined on the domain $\{(x_1,x_2,x_3)\in [0.32, 0.34]^3: x_1+x_2+x_3 = 1\}$ has maximum value $1/24$.
This proves~\eqref{equ::ex}.

Combining the lower bound and the upper bound, we have $\pi(\fCm) = 1/4$. \qedhere
\end{proof}

Finally, we are able to present the proof of Theorem~\ref{thm::main}.

\begin{proof}[Proof of Theorem~\ref{thm::main}]
For every $\l \ge 4$ not divisible by $3$,
by Construction~\ref{con::iteEdge}, we have $\pi(\Cm) \ge 1/4$. For the upper bound, we have the following claims. Recall that for a $3$-graph $\mH$, we write $\mH[t]$ for its $t$-blow-up.
\begin{claim} \label{cla::CmBlow}
For integers $\l_1, \l_2 \ge 4$, where $\l_1 \ge 2\l_2-3$ and $3\nmid \l_2$, there exists a positive integer $t$ such that $\Cm[\l_1] \subseteq \Cm[\l_2][t]$.
\end{claim}
\begin{proof}
Let $v_1v_2\ldots v_{\l_2}$ be a copy of $\Cm[\l_2]$.
Let $t$ be sufficiently large, and let $\{v_i^1,\ldots, v_i^t\}$ be the copies of vertex $v_i$ in $\Cm[\l_2][t]$.
For $1\le k \le t-2$,
we write $(v_{i_1}v_{i_2}\ldots v_{i_{j-1}} v_{i_j})^k$ for the sequence
$$
v_{i_1}^3v_{i_2}^3\ldots v_{i_{j-1}}^3 v_{i_j}^3v_{i_1}^4v_{i_2}^4\ldots
v_{i_{j-1}}^4 v_{i_j}^4\ldots v_{i_1}^{k+2}v_{i_2}^{k+2}\ldots v_{i_{j-1}}^{k+2}v_{i_j}^{k+2}.
$$
\begin{itemize}
    \item If $\l_1 \equiv 0 \pmod{3}$, then
    $(v_1v_2v_3)^{\frac{\l_1}{3}}$
    forms a copy of $\C_{\l_1} \supseteq \Cm[\l_1]$ in $\Cm[\l_2][t]$.

    \item If $\l_1 \equiv \l_2 \pmod{3}$, then
    $
    (v_1v_2v_3)^{\frac{\l_1-\l_2}{3}} v^{1}_1 v^{1}_2 v^{1}_3 v^{1}_4 \ldots v_{\l_2-1}^{1}
    v_{\l_2}^{1}
    $
    forms a copy of $\Cm[\l_1]$ in $\Cm[\l_2][t]$.

    \item If $\l_1 \equiv 2\l_2 \pmod{3}$, then
    $$
    (v_1v_3v_2)^{\frac{\l_1-2\l_2+3}{3}}
    v_1^1v_3^1v_2^1v_4^1v_3^2v_5^1 \ldots v_{\l_2-2}^1v_{\l_2-3}^2v_{\l_2-1}^1v_{\l_2-2}^2v_{\l_2}^1v_{\l_2-1}^2
    $$
    forms a copy of $\Cm[\l_1]$ in $\Cm[\l_2][t]$.
    \qedhere
\end{itemize}
\end{proof}

\begin{claim} \label{cla::fCmBlow}
For integers $\l, \L \ge 4$, where $\l \ge 2\L -3$, there exists a positive integer $t$ such that $\Cm \subseteq \F[t]$ for every $\F \in \fCm$.
\end{claim}
\begin{proof}
    Let $t$ be sufficiently large.
    Note that for every pseudo-cycle minus one hyperedge of size $\l'$, its $\l'$-blow-up contains a copy of $\Cm[\l']$. Hence, for every pseudo-cycle minus one hyperedge $\F \in \fCm$ with size $\l' \le \L$, by Claim~\ref{cla::CmBlow}, we have $\Cm \subseteq \Cm[\l'] [t/\L]
    \subseteq \F [t].$
\end{proof}

Let $\l$ and $\L$ be sufficiently large integers, where $\l \ge 2\L - 3$. Let $t$ be sufficiently large.
By Theorem~\ref{thm::turanBlowup} and Proposition~\ref{pro::main}, we have $\pi(\fCm[\le \L][t]) = \pi(\fCm) = 1/4$.
By Claim~\ref{cla::fCmBlow}, we have $\pi(\Cm) \le \pi(\fCm[\le \L][t]) = 1/4$.

Thus, we have $\pi(\Cm) = 1/4$, for every sufficiently large $\l$ not divisible by $3$.
\end{proof}

\section{Maximum number of almost similar triangles in the plane} \label{sec::tri}
In this section, we provide our new proof for Theorem~\ref{thm::triangle}.

For a triangle $\tri$, a real number $\eps > 0$, and a finite set of points $P \in \R^2$, let $\mH(P,\tri,\eps)$
be the $3$-graph on vertex set $P$, where $\{a,b,c\}$ is a hyperedge if $a,b,c$ form a triangle in $\R^2$ that is $\eps$-similar to $\tri$, and let $\mH(P,\tri)$
be the $3$-graph on vertex set $P$, where $\{a,b,c\}$ is a hyperedge if $a,b,c$ form a triangle in $\R^2$ that is similar to $\tri$.

\begin{definition} \label{def::fFtri}
Fix $C_{tri}$ to be an absolute constant.
    The \emph{forbidden family} $\fF_{tri}$ is the collection of $3$-graphs $\F$ with at most $C_{tri}$ vertices, where for almost all triangles $\tri$, there exists $\eps= \eps(\tri) >0$ such that $\mH(P,\tri, \eps)$ is $\F$-free, for every point set $P \subseteq \R^2$.
\end{definition}
The $C_{tri}$ in Definition~\ref{def::fFtri} can be chosen arbitrarily and is to make $\fF_{tri}$ a finite set to avoid some technical problems about infinity.
We let $C_{tri} \ce \L + 3$, where $\L$ is the constant in Theorem~\ref{thm::main}.
By definition, we have $h(n,\tri,\eps) = \max_{P \subseteq \R^2, |P| = n} |\mH(P,\tri,\eps)|$. Hence, $\ex(n,\fF_{tri}) \ge h(n,\tri,\eps)$ for almost all triangles $\tri$ and small enough $\eps = \eps(\tri)$.
B\'ar\'any and F\"uredi~\cite{barany2018almost} proved that several $3$-graphs, including $\K_4^-$ and $\Cm[5]$, are in $\fF_{tri}$ and then gave an upper bound for $\ex(n,\fF_{tri})$ using flag algebra. Balogh, Clemen, and Lidick\'y~\cite{balogh2022maximum} provided more members of $\fF_{tri}$ and then used a combination of flag algebra and stability method to obtain Theorem~\ref{thm::triangle}. We will prove that every tight cycle minus one hyperedge of size $4\le \l \le C_{tri}$ not divisible by $3$ is in $\fF_{tri}$, and then Theorem~\ref{thm::triangle} follows immediately from Theorem~\ref{thm::main}.

\begin{proposition} \label{pro::CminfFtri}
    For every integer $\l$, where $4 \le \l \le C_{tri}$ and $3\nmid \l$, we have $\Cm \in \fF_{tri}$.
\end{proposition}
\begin{proof}
Using the fact that every algebraic set, which is not the whole space, has measure $0$, B\'ar\'any and F\"uredi~\cite{barany2018almost} showed that in order to prove that a $3$-graph $\mH$ is in $\fF_{tri}$, we only need to prove that there exists \emph{one} triangle $\tri$ such that $\mH(P, \tri)$ is $\mH$-free, for every point set $P \subseteq \R^2$, see their proof of Lemma 9.2. See also the proof of Lemma 2.3 in~\cite{balogh2022maximum}. Therefore, denoting by $\tri_0$ the equilateral triangle, we only need to prove that $\mH(P,\tri_0)$ is $\Cm$-free, for every integer $\l$, where $4 \le \l \le C_{tri}$ and $\l$ is not divisible by $3$, and point set $P \in \R^2$. This follows from the following simple coloring argument.

Assume for contradiction that there are a point set $P \subseteq \R^2$ and an integer $\l \ge 4$ not divisible by $3$ such that $\mH(P, \tri_0)$ contains
$v_0v_1\ldots v_{\l-1}$ as a copy of $\Cm$. Without loss of generality, we can assume that $v_0 = (0,0)$, $v_1 = (1,0)$, and $v_2 = (1/2, \sqrt{3}/2)$. Let $P_0$ be the point set $\{x v_1 + y v_2: x,y\in \Z\}$. Color $P_0$ with colors in $\{0,1,2\}$ as follows. For every point $xv_1+yv_2 \in P_0$, color it with color $c \in \{0,1,2\}$, where $c \equiv x + 2y \pmod{3}$.
Note that every equilateral triangle with side length one formed by points in $P_0$ is a rainbow, i.e., its vertices have all three colors.
Now, $v_0 = (0,0)$ has color $0$ and $v_1 = (1,0)$ has color $1$.
Since $\{v_i, v_{i+1}, v_{i+2}\} \in \mH(P, \tri_0)$ for $0\le i \le \l-3$, we have, by induction, $v_i \in P_0$, $v_i, v_{i+1}, v_{i+2}$ form an equilateral triangle with side length one,
and $v_i$ has color $c$, where $c \equiv i \pmod{3}$.
Then, $v_{\l-2}, v_{\l-1}, v_0$ also need to form an equilateral triangle with side length one, so it is a rainbow. However, since $3\nmid \l$, we have that one of $v_{\l-2}, v_{\l-1}$ has color $0$, the same color as vertex $v_0$, a contradiction. \qedhere
\end{proof}

\section{Concluding remarks}
The constant $\L$ in Theorem~\ref{thm::main} given by the current proof can be large, due to the following two reasons.

For the stability result in Section~\ref{sec::Sta}, we use a regularity lemma, Theorem~\ref{thm::remLem}, which can make the dependence between $\eps_1,\eps_2$ and $\delta$ very poor in Proposition~\ref{pro::sta}. We remark that using the regularity lemma is not really necessary: we can instead use a similar averaging argument as in the proof of Claim~\ref{cla::numT5}. This would make the proof of  Proposition~\ref{pro::sta} much longer and more technical, and we still cannot make $\L$ close to $5$ (due to the reason in the next paragraph).

As mentioned at the beginning of Section~\ref{sec::Pro}, the bottleneck in our proof is about the following problem.
\begin{problem}
For a maximum $n$-vertex $\fCm$-free $3$-graph $\mH$, how many hyperedges do we need to remove to make $\mH$ free of $\ifCm$?
\end{problem}
In Lemma~\ref{lem::rem}, our bound is $O(  n^3/\sqrt{\L})$. Any improvement to this can lead to a significant improvement for the constant $\L$ in Theorem~\ref{thm::main}. We note that Lemma~\ref{lem::rem} is the only place where we need $\L$ to be large; for all other proofs, we actually only use that the forbidden family includes $\K_4^-$ and $\Cm[5]$. If it can be shown that at most $cn^3$ hyperedges are needed to be removed from every maximum $\Cm[5]$-free $3$-graph $\mH$ to make $\mH$ free of $\ifCm$, where $c$ is small enough, then the same proof gives $\pi(\Cm[5]) = 1/4$.

\section*{Acknowledgment}
The authors are grateful to Bernard Lidick\'y and Zolt\'{a}n F\"{u}redi for helpful communication. We would also like to thank the anonymous
referees for their helpful comments.


\end{document}